\documentclass[12pt]{amsart}
\usepackage{amsmath, amsthm, amscd, amsfonts, amssymb, graphicx, color, mathabx, tikz, caption, enumitem, mathtools, comment, pgfplots,amsrefs}
    \pgfplotsset{compat=1.16}
    \usepgfplotslibrary{fillbetween}
    \usetikzlibrary{quotes,angles} 
    \makeatletter
    \def\l@subsection{\@tocline{2}{0pt}{2.9pc}{5pc}{}}
    \def\l@subsubsection{\@tocline{2}{0pt}{5pc}{7.5pc}{}}
    \makeatother
\usepackage[margin=2.8cm]{geometry}

\usepackage[bookmarksnumbered, colorlinks, plainpages]{hyperref}
    \hypersetup{colorlinks=true, linkcolor=blue, citecolor=magenta, filecolor=magenta, urlcolor=cyan}
\allowdisplaybreaks
    \mathtoolsset{showonlyrefs,showmanualtags} 
\numberwithin{equation}{section}
\newtheorem{theorem}[equation]{Theorem}
\newtheorem{lemma}[equation]{Lemma}
\newtheorem{proposition}[equation]{Proposition}
\newtheorem{corollary}[equation]{Corollary}
\theoremstyle{definition}

\newtheorem{conjecture}[equation]{Conjecture}

\theoremstyle{remark}

\newcommand{\N}{\mathbb{N}}

\newcommand{\C}{\mathbb{C}}
\newcommand{\Z}{\mathbb Z}

\newcommand{\A}{\mathbb{A}}

\DeclareMathOperator{\Hi}{Hi} 
\DeclareMathOperator{\Lo}{Lo} 
 
\DeclareMathOperator{\lcm}{lcm}

\begin{document}
\title[Gaussian Primes: Goldbach's Conjecture and Improving Estimates]{Averages over the Gaussian Primes: Goldbach's Conjecture and Improving Estimates}

\author[Giannitsi]{Christina Giannitsi}
    \address{School of Mathematics, Georgia Institute of Technology, Atlanta GA 30332, USA}
    \email{cgiannitsi3@math.gatech.edu}
    %\thanks{Research supported in part by grant  from the US National Science Foundation, DMS-1949206.}
\author[Krause]{Ben Krause}
    \address{Department of Mathematics, King’s College, London, WC2R 2LS, UK}
    \email{ben.krause@kcl.ac.uk}
\author[Lacey]{Michael T. Lacey} 
    \address{ School of Mathematics, Georgia Institute of Technology, Atlanta GA 30332, USA}
    \email {lacey@math.gatech.edu}
    %\thanks{MTL: The author is a 2020 Simons Fellow. Research of all authors supported in part by grant  from the US National Science Foundation, DMS-1949206}
\author[Mousavi]{Hamed Mousavi}
    \address{School of Mathematics, Georgia Institute of Technology, Atlanta GA 30332, USA}
    \email{hmousavi6@gatech.edu}
\author[Rahimi]{Yaghoub Rahimi}
    \address{School of Mathematics, Georgia Institute of Technology, Atlanta GA 30332, USA}
    \email{yaghoub.rahimi@gatech.edu}

\begin{abstract}  
We prove versions of Goldbach conjectures for Gaussian primes in arbitrary sectors.   
Fix an interval $\omega \subset \mathbb{T}$. There is an integer $N_\omega $, so that 
every odd integer $n$ with $N(n)>N_\omega $ 
and $\textup{dist}( \arg(n) , \mathbb{T}\setminus \omega ) > (\log N(n)) ^{-B}$, 
is a sum of three Gaussian primes $n=p_1+p_2+p_3$, with 
$\arg (p_j) \in \omega  $, for $j=1,2,3$. 
A density version of the binary Goldbach conjecture is proved. 
Both follow from a  High/Low decomposition of the Fourier transform of averages over Gaussian primes, defined as follows.  
Let $\Lambda(n) $ be the Von Mangoldt function for the Gaussian integers and consider the norm function $N:\Z[i]\rightarrow \Z^+$, $\alpha + i \beta \mapsto \alpha ^2 + \beta ^2$. Define the averages 
$$A_Nf(x)=\frac{1}{N}\sum_{N(n)<N}\Lambda(n)f(x-n).$$
Our decomposition also proves the   $\ell ^p$ improving estimate
$$ 
\lVert A_N f \rVert _{\ell^{p'}}\ll   N ^{1/p'- 1/p} \lVert f\rVert _{\ell^p}, 
\qquad 1<p\leq 2., 
$$
 \end{abstract}

\maketitle 
\tableofcontents

\bigskip

\section{Introduction}
The principle goal of this paper is to establish a density version of the strong Goldbach conjecture for Gaussian integers, restricted to sectors in the complex plane. 

Briefly, for  integers $n \in \Z[i]$ we write $N(a+ib) = a^2 + b^2$, 
and if we express $ n = \sqrt {N(n)}e^ { i \theta }$, we set $\arg (n) = \theta $; the units are $\pm1$ and $\pm i$.  
The field $\Z[i]$ is a Euclidean domain, inside of which an integer $p= a+ib$ is a prime if it has no integer factor $q$ with $N(q) < N(p)$.  
The Gaussian primes can take one of two forms. 
First if $a$ and $b$ are non-zero, then  $N(p) = a^2 + b^2$ is prime. 
Second, if $a$ or $b$ are zero, then it is of the form of a unit times 
 a prime in $\N$ congruent to $3 \mod 4. $

As in the case of $\Z$, the Goldbach conjecture over $\Z[i]$ requires a notion of evenness:
a Gaussian integer $x = a+ib$ is \emph{even} if and only if $N (x) = a^2 + b^2$ is even.  Evenness is equivalent to either condition below. 
\begin{enumerate}
    \item  The integer $a+b$ is even. 
    \item  The  Gaussian integer $1+i$ divides $x$.  
\end{enumerate}
An integer which is not even is \emph{odd}. 

The main result concerns Goldbach type results.  
We show that there are very few even integers which are \emph{not} a sum of two primes. And, we do so even  with significant restrictions on the arguments of the integers involved. Similarly, all sufficiently large  odd integers are a sum of three primes.  
  The only  prior results in this direction we could find  in the literature correspond to the entire complex plane.  Below, we allow arbitrary sectors.  

 \begin{theorem} \label{t:Goldbach}  Fix an integer  $B>10$  and interval $\omega\subset \mathbb{T}$.  There exists an $N_{\omega, B} >0$, such that for all integers $N> N _{\omega, B}$, the following holds.  

 \begin{enumerate}
    \item  Every odd integer $n$ with $N(n)> N_{\omega , B}$ and $ \textup{dist}( \arg(n) , \mathbb{T}\setminus \omega ) > (\log N(n)) ^{-B}$, 
    is a sum of three Gaussian primes  $n= p_1 + p_2+p_3$, with $ \arg(p_j) \in \omega $ for $j=1,2,3$.  

     \item 
     We have $\lvert E_2(\omega , N) \rvert \ll \frac{ N}{\log(N)^B}$, 
     where 
     $E_2(\omega ,N)$ is the set of   \emph{even} Gaussian integers with $
 N(n)<N $ and  $  \arg(n) \in \omega$, 
 which \emph{cannot} be represented as sum of two Gaussian primes  $n= p_1 + p_2$, with  $ \arg(p_j) \in \omega $ for $j=1,2$.  
    
\end{enumerate}
\end{theorem}

In the ternary case, note that we require that the odd integer be sufficiently far from the boundary of the sector. For otherwise, we run up against issues related to the Landau conjecture on primes of the form $n ^2 +1$.  In the binary case, the statement is the standard formulation, due to the presence of the exceptional case.

Our proof derives from a detailed study of the analytic properties of averages over the Gaussian primes. 
In particular, let $f:\Z[i]\rightarrow \mathbb{C}$ be an arithmetic function, and for $x\in \Z[i]$ define
\begin{align*}
A_Nf(x)=\frac{1}{N}\sum_{N(n)<N} \Lambda(n)f(x-n)
\end{align*}
where the von Mangoldt function on $\Z[i]$ is defined by 
\begin{align*}
 \Lambda(n)=\begin{cases}
\log(N(\rho)) & \text{ if } n=\rho^{\alpha} \text{ and } \rho\in \Z[i] \text{ is irreducible} \\
0 & \text{otherwise.}
\end{cases}
\end{align*}

 We prove an improving estimate for the averages above: The averages of $\ell^1 (\Z^2)  $ 
functions are nearly bounded.  

\begin{theorem} \label{t:fixedscale} 
For all $N$, and $1< p \leq 2$, we have, 
\begin{align*}
\langle A_N f,g\rangle \ll N ^{1- 2/p}  \lVert f \rVert_p \lVert g \rVert_p. 
\end{align*}
Equivalently, for all $1 < p \leq 2$, whenever $f$ is supported on a cube, $Q$, of side length $\sqrt{N}$,
\begin{align*}
\frac{ \| A_N f \|_{p'}}{|Q|^{1/p'}} \ll \frac{ \| f \|_{p}}{|Q|^{1/p}}  
\end{align*}
\end{theorem}

To prove this, we need to develop many expected results, including Ramanujan type identities, and the Vinogradov type estimates for the Fourier transform of $A_N$.  
Not being able to identify clear cut references for many of these estimates, we develop them below. 
We then develop a High/Low decomposition of the Fourier transform of $A_N$. 
A particular innovation is our approach to the major arc estimate in Lemma \ref{majorarcestimate}, which are  
typically proved by Abel summation. 
But that method is poorly adapted to the question at hand. We develop 
a different method.   Also noteworthy is that the Low term is defined in terms of smooth numbers, a technique used in \cite{210110401L}. 
Smoothness facilitates the proof of Lemma \ref{l:HLConstantlemma}. 
With this decomposition in hand, the deduction of the Theorems is relatively straight forward.

Indeed, we develop this for the more specialized averages, over sectors of the complex plane.  
For an interval $\omega \subset \mathbb{T}$, extend the definition of the averages to this setting.  
\begin{align} \label{e:omega}
A_N ^{\omega }f(x)=\frac{2 \pi } { \lvert \omega  \rvert N} \sum_{\substack{N(n)<N \\ \arg(n) \in \omega  }} \Lambda(n)f(x-n). 
\end{align}
To address the binary Goldbach question, note that 
 \begin{equation}
 A_N ^{\omega } \ast A_N ^{\omega } (m) = 
 \frac{4 \pi^2 } { \lvert \omega  \rvert ^2  N ^2 } 
 \sum _{ \substack{ N(n_1), N(n_2) < N  \\  \arg (n_1) , \arg (n_2) \in \omega  \\ n_1 + n_2 =m}} 
 \Lambda (n_1) \Lambda (n_2)  
 \end{equation}
 The High/Low decomposition for the $A_N ^{\omega }$ can be leveraged to study the sum above.  This is the path we follow 
 to prove Theorem \ref{t:Goldbach} in \S\ref{s:Goldbach}.

\bigskip 
Our motivations come from number theory and analysis. 
  In this setting, 
the binary  (and higher order) 
Goldbach conjectures has been addressed  in the number field setting. 
Mitsui \cite{MR136590}*{\S11,12} addresses the binary and higher order cases of Goldbach's Conjecture in an arbitrary number field. 
 Mitsui finds that  all sufficiently large totally positive odd integers are the sum of odd totally positive primes.   
 He does not address the sector case.  
 Much earlier  work of Rademacher \cites{MR3069422,MR3069434}. 
Rademacher \cite{MR3069422}  studied representations of totally positive integers in a class of quadratic extensions of $\mathbb Z$.  
These papers assume the Generalized Riemann Hypothesis. 
 
Holben and Jordan \cite{MR238760} raised conjectures in the spirit of our Theorem.  Their Conjecture F states 

\begin{conjecture} Each even Gaussian integer  $ n$ with $N(n)>10$ is the sum of two primes $n=p_1+p_2$, 
with  $\arg(n)- \arg(p_j) \leq \pi /6$, for $j=1,2$.
\end{conjecture}

As far as we know, there is no prior result in this direction, for neither the binary nor the ternary version of the Goldbach conjecture.  But, we also note that our Theorem \ref{t:Goldbach} 
provides a density version of a much stronger result. 
For all $\delta >0$,  most even Gaussian integers $n$ with $N(n)> N_\delta $, are the sum of two primes $n=p_1+p_2$, 
with $\lvert \arg(n)- \arg(p_j) \rvert \leq \delta $, for $j=1,2$.  
Indeed, partition the unit circle into intervals of length $\delta $, and apply our Theorem to each interval. The ternary form of their conjecture holds for large odd integers $n$, seen by covering the unit circle by overlapping intervals of diameter $\pi /6$.

\medskip  
The study of metric properties of the uniform distribution of $p \alpha$, 
for irrational $\alpha \in \mathbb C $, and Gaussian prime $p$,  is well developed 
by Harman \cite{MR2331072}*{Chapters 11}, \cite{MR4045111}, with effective results even 
small sectors.  
See the extensions to certain quadratic number fields by Baier and Technau \cite{MR4206429}. The latter paper also addresses metric questions along lines in the complex plane.

 On the analytical side, our improving estimate above is part of Discrete Harmonic Analysis, a subject invented by Bourgain \cite{MR1019960}. The recent textbook of one of us \cite{MR4512201} serves as a comprehensive summary of the subject. The study of the averages over the primes has 
 been extensively studied, beginning with the work of Bourgain-Wierdl \cites{MR995574,BP1}, and continued by several 
 \cites{MR4434278,MR3370012,MR3299842,MR3646766,MR4072599}.

\bigskip 
\section{Notation and Preliminaries}

Throughout the paper we are using  $\| \beta\|$ to denote the distance of $\beta\in \mathbb{C}$ to the closest point $n\in \Z[i]$, and for $a,b,c \in \Z[i]$ we write $(a,b)=c$ to indicate that $c$ is the greatest (in norm) common divisor of $a$ and $b$ up to a unitary element.

We will  use the $\ell^\infty$ balls 
 $B_\infty(r) = \{ x = a+bi \in \Z [i] \, : \, -r \leq a,b  < r \}$ and $B_\infty(c,r) = B_{\infty}(r) +c$. Notice that there is a small departure from the traditional notation of a ball with respect to the $\infty$-norm, in the sense that we are only including the lowest endpoint, however this variation is useful as it allows us to obtain a tessellation of the complex plane. It is obvious that translating and rotating the grid does not affect the tessellation. For a $q = |q| e^{i \theta_q} \in \Z [i]$, we are interested in the grid formed by the squares
$$ B_q : = e^{i \theta_q}  B_\infty \left((1+i)\frac{|q|}{2}, \frac{|q|}{2} \right) $$
where we have rotated the squares by the argument of $q$, so that their sides are parallel and orthogonal to $q$ respectively; see Figures \ref{square_Sq} and \ref{square_Sq_2}. This particular decomposition of lattice points coincides with all the \emph{unique} remainders modulo $q$ in $\Z[i]$ and simplifies our calculations of complex exponentials.
Indeed, it is known that for $q \in \Z[i] $, there are $N(q) $ distinct remainders  modulo $q$, which agrees with the number of points inside $B_q$. It is straightforward to verify that any two points in $B_q$ cannot be equivalent modulo $q$, which then proves our claim. We are also going to need the following, more geometric, description of our boxes: 
\begin{equation}\label{eq-geometricsquaredefinition}
    B_q = \{r \in \Z[i] \mid 0 \leq \langle r , q\rangle < N(q) \ \ \text{and} \ \ 0 \leq \langle r , iq\rangle < N(q) \}.
\end{equation}

\begin{figure}
    \begin{minipage}[b]{0.45\linewidth}
    \hspace*{2em} 
    \begin{tikzpicture}
        \draw[->, black] (-1,0) -- (3,0);
        \draw[->, black] (0,-0.5) -- (0,3);
        \draw[blue, very thick]  (0,0) -- (2,1) ;
        \draw[blue, very thick]  (0,0) -- (-1,2) ;
        \draw[dashed,blue, very thick]  (2,1) -- (1,3) -- (-1,2) ;
        \draw[black] (2,1) circle (2pt) node[anchor=west]{$q $};
        \draw[black] (-1,2) circle (2pt) node[anchor=east]{$iq $};
        \draw[black] (1,3) circle (2pt);
        \filldraw[black] (0,0) circle (2pt);
        \filldraw[black] (0,1) circle (2pt);
        \filldraw[black] (0,2) circle (2pt);
        \filldraw[black] (1,1) circle (2pt);
        \filldraw[black] (1,2) circle (2pt);
    \end{tikzpicture}
    \caption{\tiny \\ The lattice points in the box $B_q $, for $q = 2+i $. There are exactly $N(q) = 5$ points inside the box. \\ ~ }
    \label{square_Sq}
    \end{minipage} 
    \hfill 
    \begin{minipage}[b]{0.45\linewidth}
    \begin{tikzpicture}
        \draw[->, black] (-1,0) -- (3,0);
        \draw[->, black] (0,-0.5) -- (0,6);
        \draw[blue, very thick]  (0,0) -- (4,2) ;
        \draw[blue, very thick]  (0,0) -- (-2,4) ;
        \draw[dashed,blue, very thick]  (4,2) -- (2,6) -- (-2,4) ;
        \draw[dashed, gray, thick]  (2,1) -- (0,5);
        \draw[dashed, gray, thick]  (-1,2) -- (3,4);
        \draw[black] (4,2) circle (2pt) node[anchor=west]{$q $};
        \draw[black] (-2,4) circle (2pt) node[anchor=east]{$iq $};
        \draw[black] (2,6) circle (2pt);
        \draw[black] (0,5) circle (2pt);
        \draw[black] (3,4) circle (2pt);
        \filldraw[black] (-1,2) circle (2pt) node[anchor=east]{$\frac{iq}{2} $};
        \filldraw[black] (-1,3) circle (2pt);
        \filldraw[black] (-1,4) circle (2pt);
        \filldraw[black] (0,0) circle (2pt);
        \filldraw[black] (0,1) circle (2pt);
        \filldraw[black] (0,2) circle (2pt);
        \filldraw[black] (0,3) circle (2pt);
        \filldraw[black] (0,4) circle (2pt);
        \filldraw[black] (1,1) circle (2pt);
        \filldraw[black] (1,2) circle (2pt);
        \filldraw[black] (1,3) circle (2pt);
        \filldraw[black] (1,4) circle (2pt);
        \filldraw[black] (1,5) circle (2pt);
        \filldraw[black] (2,1) circle (2pt) node[anchor=north]{$\frac{q}{2} $}; 
        \filldraw[black] (2,2) circle (2pt);
        \filldraw[black] (2,3) circle (2pt);
        \filldraw[black] (2,4) circle (2pt);
        \filldraw[black] (2,5) circle (2pt);
        \filldraw[black] (3,2) circle (2pt);
        \filldraw[black] (3,3) circle (2pt);
    \end{tikzpicture}
    \caption{\tiny \\ The box $B_q $ for $ q = 4 + 2i$. Note that $ 4+2i = 2(2+i)$, hence $B_{4+2i}$ is partitioned by $2^2$ Gaussian translations of $ B_{2+i}$.}
    \label{square_Sq_2}
    \end{minipage} 
    \hspace*{2em} 
\end{figure}

\smallskip 

Let $e(x) := e^{2 \pi i x}$.
For a function $f:\Z[i] \to \mathbb{C}$, we can define the discrete Fourier Transform over the box $B_q$ as
\begin{align} \label{e:FDq}
    \mathcal F_{q} f \, (x) := \sum_{n\in B_q} f(n) e\bigl(-\langle x,\frac{n}{q}\rangle\bigr),
\end{align}
and the corresponding inverse discrete Fourier transform as
\begin{align*}
    \mathcal F_{q}^{-1} f \, (n) := \frac{1}{N(q)}\sum_{x\in B_q} f(x) e\bigl(\langle n,\frac{x}{\bar{q}}\rangle\bigr).
\end{align*}
\smallskip

The Euler totient function $\phi$  for Gaussian integers counts the number of points in $B_q$ that are coprime to $z$.  It is equal to 
\begin{align*}
    \phi(z) & = N(z) 
        \prod_{\substack{p \divides z \\ p : \text{ prime in }  \Z[i]}} \left( 1 - \frac{1}{N(p)} \right). 
\end{align*}
Clearly, $\phi (z) = \phi ( \bar z)$. 
Thus, $ \lvert \A_q \rvert = \phi (q)$, where 
\begin{equation}
    \A_q : =  \{ r \in B_q \, : \, (r,q)=1 \}. 
\end{equation}

We also need the arithmetic function $r_2(n)$, which is called \emph{Sum of Squares function}, which counts the number of representations of an integer $n$ as sum of two squares. Obviously $r_2(n)$ is the number of $m\in \Z[i]$ such that $N(m)=n$. It is known that  $r_2(n)=O(1)$ in an average sense, namely  
\begin{align}\label{sumofsqestimate}
%\sum_{n<N}\frac{r_2(n)}{n}&\simeq \log N &
\sum_{n<N}r_2(n)&\simeq N
\end{align}

% For any normal subgroup  $m\lhd \Z[i]$, the homomorphism 
% $$\chi\coloneqq \left(\frac{\Z[i]}{m} \right)^*\rightarrow \mathbb{C}^*$$
% is a Dirichlet character. \textcolor{red}{We do not need the Gr\H{o}ssencharakteren definition. We can simply put the Dirichlet's character here. I highlight the unnecessary definitions.} 
% \color{red}
% Let $\chi_{\infty}:\mathbb{C}^*\rightarrow S^1$ be defined as $\chi_{\infty,l}  (z)=(\frac{z}{|z|})^l$ for a choice of  $l\in \Z$. Note that if $\arg(z)$ is a rational number, then 
% $\chi_{\infty,l} $ takes finitely many values.
% A Hecke character modulo $q$ is defined as  
% $$\psi(\langle \alpha\rangle)=\chi(\alpha)\chi_{\infty}(\alpha).$$
% where $\langle \alpha\rangle$ is the ideal generated by $\alpha$ in the ring $\Z[i]$.
% \textcolor{blue}{(What is $\chi _ \infty$?  That is not defined up to this point. Are we really using these characters?)}

% Also, a Gauss sum is   
% \begin{align*}
% \tau(\psi,a) = \sum_{\substack{N(b)<N(m)\\(b,m)=1}}\psi(b)e(\langle a,\frac{b}{q}\rangle),
% \end{align*}
% where  $e(x)=e^{2\pi ix}$. 
% Note that, for our purposes, we only need the principal Gauss sum, 
% that is the Gauss sum with 
% $\psi \equiv 1$. In this case, the Gauss sum is analogous to a Ramanujan sum, up to normalization. 
% For $(a,q)=1$, it is known that $\tau(\psi,a) = \psi(a)\tau(\psi,1)$ and $|\tau(\psi,a)|\leq\sqrt{N(q)}$. 

Define the  M\"obius function as follows. 
\begin{align} \label{e:Mobius}
 \mu(n)=\begin{cases}
(-1)^k & \text{ if } n=\epsilon \rho_1\rho_2\cdots \rho_k\\
1 & \text{ if } n=\epsilon\\
0 & \text{otherwise.}
\end{cases}
\end{align}
where $\epsilon\in\{\pm 1,\pm i\}$ are units and $\rho_i$ are  distinct prime elements of $\Z[i]$.
%These are averages defined over the set of Gaussian prime numbers $\mathcal P_N$,  with $N(x)<N$, and their respective powers.

\smallskip 
We define the classical form of the average over a sector.  
For an interval $\omega \subset \mathbb{T}$, set 
\begin{equation}  \label{e:MN}
M_N ^{\omega } = M_N = \frac{\lvert\omega\rvert}  { 2 \pi  N} \sum 
_{ \substack { N(n) < N \\ \arg(n)\in\omega }} \delta _n , 
\end{equation}
where $\delta _n$ is a Dirac measure at $n$.   

We emphasize that \emph{we do not attempt to track constants that depend upon $\omega$.} 
Frequently, we assume that $N$ is large enough, once $\omega $ is fixed.

\begin{lemma}\label{expsum}
Fix an interval $\omega \subset \mathbb{T}$, $0< |\omega| \leq  2 \pi $. 
For $N(\beta)<1$, and integers $N> N _{\omega }$, 
\begin{align}
    \widehat M_N ^{\omega }(\beta) & \coloneqq  
 \frac{\lvert\omega\rvert}  { 2 \pi  N} \sum _{ \substack { N(n) < N \\\arg(n)\in\omega  }}
    e(-\langle n,\beta\rangle) 
    \\ \label{e:expsum}
    & \ll_{\omega}  \min\left(1, (N \cdot N(\beta))^{-\frac{3}{4}}\right). 
    + \frac1{\sqrt N }.
\end{align}
The implied constant only depends on $\omega$. 
%(for $N _{\omega }$ sufficiently large). 
\end{lemma}

\begin{proof}
Let $I_n = B_\infty (n, 1/2)$. For $n=0$, we have 
\begin{equation}
    \int_{I_0} e (-\langle x, \beta\rangle) \;dx = 
    \prod _{j=1}^2 \frac{\sin (\pi \beta_j)} {\pi \beta_j}, \; \; \; \beta = (\beta_1,\beta_2). 
\end{equation}
This function is bounded above, and away from zero, since 
$N(\beta) <1$.  Suppress the dependence on $\omega $ in the notation. 
 Modify the definition of 
 \[ S_N (\beta) := N \widehat M_N (\beta ) \]
 to 
\begin{align}
 T_N (\beta) &\coloneqq     \sum _{\substack{n \colon N(n)< N \\ \arg(n)\in\omega    }} 
    \int_{I_n} e ( -\langle x, \beta \rangle )\; dx 
   \\& = \prod _{j=1}^2 \frac{\sin (\beta_j/2)} {\beta_j/2} 
   \sum _{\substack{ n\colon N(n)< N \\ \arg(n)\in\omega }} e(-\langle n, \beta \rangle)
    \\
    &= S_N (\beta) 
    \prod _{j=1}^2 \frac{\sin (\beta_j/2)} {\beta_j/2}   . 
    \end{align} 
  We see that it suffices to estimate $T_N(\beta)$. 
    
The symmetric difference between the sector defined by $ \omega$, and the area of integration defining $T_N (\beta) $ is the set 
\begin{equation}
    \bigcup _{\substack{ n\colon N(n)< N \\ \arg(n)\in\omega }}  I_n \, 
    \triangle  \{ n\colon N(n)<N,\arg(n)\in\omega  \} .
\end{equation}
It has measure at most $\ll  \omega \sqrt N$, as the above set is supported in an $O(1)$ neighborhood of the boundary $\{ n : N(n) < N, \ n \in \omega\}$, which has length $\omega \sqrt{N}$. Thus, 
\begin{align} 
    T_N(\beta) &\ll 
    \int_{ \substack {N( x ) <N \\ \arg (x)\in\omega } }  
    e(-\langle x, \beta\rangle ) \; dx +  \sqrt{N}  
%    \\ &    
%    \ll   N \min \Bigl\{ 1,   \frac { 1}{ [N (\beta )  \sqrt{N}]^{1/2}} %\Bigr\} + \sqrt{N}.  \textcolor{red}{\text{ Please check here. I am sure the lemma is correct though.}}  
\end{align}
By partitioning $ \omega$ into smaller arcs and arguing as in \cite{MONTGOMERY}*{Page 111-112}, which addresses the case where $\omega = \mathbb{T}$, we may bound the integral
\[
\frac{1}{N} \int_{ \substack {N( x ) <N \\ \arg (x)\in\omega } }  
    e(-\langle x, \beta\rangle ) \; dx \ll_{\omega} \big( \frac{1}{N \cdot N(\beta)} \big)^{3/4}.
\]

%Concerning the integral, it is largest when e.g.{} $\omega$
%is centered on the positive $x$-axis, and $\beta = (0, \beta_2)$.  (So $\beta$ is as orthogonal to a typical value of $x$ as it can be.) 
%Then, with $\rho = \lvert \omega \rvert /2$ and $\beta_2\neq 0$, 
%\begin{align}
%  \int_{ \substack {N( x ) <N \\ \arg (x)\in\omega } }  
%    e(-\langle x, \beta\rangle ) \; dx 
%    & = \int_{-\sqrt{N}}^{\sqrt{N}} \int _{-x_1\tan(\rho) }^{x_1\tan(\rho)} e(- \beta_2 x_2)\;dx_2 \, dx_1 
%    \\ 
%  & = \frac{-i}{\pi \beta_2} 
%\int_{-\sqrt{N}}^{\sqrt{N}}  \sin (2 \pi \beta_2  \tan(\rho) x_1) \; d x_1 
%\\ 
%& = \frac{-i}{2\pi^2 \beta_2 ^2  \tan(\rho)} 
%\int_{-2\pi \sqrt{N}\beta_2\tan(\rho)}^{2\pi \sqrt{N}\beta_2\tan(\rho)} \sin (x_1) \; d x_1
%\\ 
%& \ll [ N(\beta)   \lvert \omega \rvert ]^{-1}.  \end{align}

%\color{red}
%I guess $|x_1|\leq \sqrt{N}$ and 
%$$x_2=x_1\tan(\arg(x))\in [-x_1\tan(\rho),x_1\tan(\rho)]$$
%\color{blue} Just make the changes yourself. 
%\color{black}
%\begin{align}
%  \int_{ \substack {N( x ) <N \\ \arg (x)\in\omega } }  
%    e(-\langle x, \beta\rangle ) \; dx 
%    & = \int_0^N \int _{-N \rho x_1}^{N\rho x_1} e(- \beta_2 x_2)\;dx_2 \, dx_1 
%    \\ 
%  & = \frac{1}{\pi \beta_2} 
%\int_0^N  \sin (2 \pi \beta_2 N \rho x_1) \; d x_1 
%\\ 
%& = \frac{1}{2\pi^2 \beta_2 ^2 N \rho} 
%\int_0^{\beta_2 N^2 \rho } \sin (x_1) \; d x_1
%\\ 
%& \ll [ N(\beta) N  \lvert \omega \rvert ]^{-1}.  \end{align}
 %The implied constant is absolute.
% We see that our estimate holds,  provided $N$ is sufficiently large, as a function of $\omega $.   This completes the proof. 
\end{proof}
 
This is an elementary calculation. 
\begin{proposition}  Let $ \omega \subset \mathbb T$, and $ N(\beta ) \gg \lvert \omega \rvert  ^{-2}$. We have 
\begin{equation} \label{e:FourierDecay}
 \frac {2 \pi } {\lvert \omega\rvert}  \int _{ \substack{ N(x) \leq 1 \\  \arg (x)\in\omega }  } e(- \langle x, \beta  \rangle) 
\ll    \frac { \ln N( \beta )}{ N (\beta )^{1/2}} 
\end{equation}
\end{proposition}

\begin{proof} Write the integral as 
 \begin{align}
 \frac {2 \pi } {\lvert \omega\rvert}  \int_0^ \omega \int_0 ^1 e ( \langle \beta , t e (  \theta ) \rangle )\; t dt\, d \theta  
 &\ll 
 \frac 1  {\lvert \omega\rvert}  \int_0 ^{\omega  } \min \{1,  \lvert \langle \beta , e (\theta ) \rangle \rvert ^{-1} \} d \theta  
\\ & \ll \frac 1 {\lvert \omega\rvert}  \frac { \ln N( \beta )}{ N (\beta )^{1/2}} .
 \end{align}
% The last integral is divided up into regions for which $ \langle \beta , e (\theta ) \rangle  $ is approximately constant. 
\end{proof}

 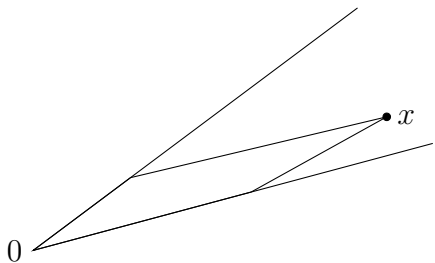
\begin{figure}
 \begin{tikzpicture}[rotate=15] 
  \draw
    (5.5,0) coordinate (a)    
    -- (0,0) coordinate (b) node[left] {0}
    -- (5,2) coordinate (c) 
    pic[angle eccentricity=1.2, angle radius=1cm]
    {angle=a--b--c};
    \filldraw (5,.5) circle (.05cm) node[right] {$x$}; 
    \draw (0,0) -- (1.5,.6) -- (5,.5) 
    -- (3,0) -- (0,0) ; 
\end{tikzpicture}
\caption{A point $x$ is in the sector $S _{\omega }$. The parallelogram $T^\omega (x)$ is pictured inside the sector.} 
\label{f:convolve}
 \end{figure}

This is needed for the Goldbach Theorems. 

\begin{lemma}\label{l:convolve}  There is a constant $C>1$ so that this holds. 
Fix interval $\omega \subset \mathbb{T} $ of length less than $\pi /4$.   For $N > N( \omega )$,  
we have this estimate.  and $\textup{arg} \in \omega $, 
with $ \delta (x) \coloneqq \textup{dist}(\textup{arg}(x), \mathbb{T} \setminus \omega ) > C/\sqrt N $, we have 
\begin{align}
M_N ^{\omega } \ast M_N ^{\omega } (x) &\gg   \frac{ \delta (x)} N , 
\\ 
M_N ^{\omega } \ast M_N ^{\omega } \ast M_N ^{\omega }  (x) &\gg   \frac{ \delta (x) ^2 } N , 
\end{align}
\end{lemma}

\begin{proof} Consider the continuous case,  
as dilation invariance lets us reduce to the case of $N=1$. 
Set 
\begin{equation}
 S^\omega \coloneqq \{ x\in \mathbb{R} ^2  \colon  N(x) \leq 1,\  \textup{arg}(x) \in \omega  \}. 
\end{equation}
Let  $x \in S ^{\omega }$ be 
of norm 1.  Consider the set 
\begin{equation}
T^\omega (x) = \{ y \in S ^ \omega \colon  \mathbf{1} _{S^\omega } (x-y)=1\}. 
\end{equation}
This set is a parallelogram, with one side length at least $1/2$, and  width comparable to $\delta (x)$. 
See Figure \ref{f:convolve}. 
In particular $\lvert T ^\omega (x) \rvert \gg \delta (x)$. 
But then, 
\begin{equation}
\mathbf{1} _{S ^{\omega }} \ast \mathbf{1} _{S ^{\omega }} (x) = \lvert T ^\omega (x) \rvert  \gg \delta (x). 
\end{equation}
 By dilation invariance, this means 
\begin{equation}
\mathbf{1} _{S ^{\omega }} \ast \mathbf{1} _{S ^{\omega }} (x) \gg N(x) \delta (x), \qquad x\in S ^{\omega }. 
\end{equation}
Note that the third convolution is then 
\begin{equation}
\mathbf{1} _{S ^{\omega }} \ast\mathbf{1} _{S ^{\omega }} \ast \mathbf{1} _{S ^{\omega }} (x) 
= 
\int \lvert T ^\omega (x-y)  \rvert \mathbf{1} _{S ^{\omega }} \ast \mathbf{1} _{S ^{\omega }} (y)\; dy 
\gg   \delta (x) ^2 , \qquad N(x)=1,
\end{equation} 
since for typical values of $x$ and $y$, both expressions above are $\gg \delta (x)$.   Note that one has, by dilation invariance, 
\begin{equation}
\mathbf{1} _{S ^{\omega }} \ast\mathbf{1} _{S ^{\omega }} \ast \mathbf{1} _{S ^{\omega }} (x) 
\gg N(x) ^2 \delta(x) ^2, \qquad x\in S^\omega.     
\end{equation}

\medskip
Turn to the discrete setting.  We no longer have dilation invariance. The situation is this. 
Let $N$ be a large integer, $\textup{arg}(x) \in \omega $, $ N(x) \geq \tfrac 1{400} N$, and $\delta (x) \sqrt N > C_0$. 
Here, $C_0$ is a large enough constant. 
Consider the parallelogram $T_x \coloneqq \sqrt N \cdot T ^{\omega }(x)$.   
The key point is that   the number of lattice points in $T(\omega,x)$ is comparable to its area.  
That is,  
\begin{equation} \label{e:interior}
\sum _{n \in T(\omega,x) \cap \mathbb{Z} [i]} 1 \gg \lvert T (\omega,x) \rvert 
\gg \delta(x) N. 
\end{equation}
Indeed, $T(\omega,x)$ is a parallelogram, with side length $\gg \sqrt N$ and width comparable to  $ \gg \textup{arg}(x) \sqrt N  \gg  C_0$.  
It is then elementary to see that \eqref{e:interior} holds.  
This is the discrete variant of the key estimate in the continuous case. We are free to repeat the remainder of the continuous argument to conclude the Lemma. 

The details are as follows.  
In the bilinear case, we have, for $N(x) \geq \tfrac1{400} N$,  and $\textup{arg}(x)\in \omega$, we have 
\begin{align}
    M^\omega_N \ast M^\omega_N   (x) 
    &\geq 
    N^{-2} 
    \sum _{n \in T(\omega,x) \cap \mathbb{Z} [i]} 1 
    \\ &\gg \lvert T (\omega,x) \rvert   \gg \delta(x) N ^{-1}, 
\end{align}
where we use \eqref{e:interior}.   This estimate is stronger than claimed above, as it helps with the ternary estimate. 
Then, requiring that $ N(x) > \tfrac{3}{4} N$, 
note that 
\begin{align}
    M^\omega_N \ast M^\omega_N \ast M^\omega_N   (x) 
    \geq 
    N^{-1}  
    \sum _{n \in T(\omega,x) }  M^\omega_N \ast M^\omega_N   (x-n) \gg \delta (x)  ^{2} N ^{-1}. 
\end{align}
This is because for typical  values of $n\in T(\omega,x)$, 
we have $ M^\omega_N \ast M^\omega_N   (x-n) \gg \delta (x) N ^{-1}$.
\end{proof}

\section{Inequalities involving Ramanujan Sums}

In this section, we prove two dimensional analogues of estimates and identities already known for one dimensional Ramanujan sums, like the Cohen Identity. This section is crucial to our High-Low decomposition.
Some of these properties are known, but we include details for completeness. We start with standard facts about Fourier transform on $B_q$. 

\begin{lemma} \label{orthoidentity-lm}
Consider the set $B_q $ and $n \in \Z[i] $. We have:
\begin{align}
\label{e:orthognality}
    \sum_{r\in {B_q}} e( \langle r,\frac{n}{\bar q}\rangle) = \begin{cases}
    N(q) & \textup{ if } \bar q \divides n \\
    0 & \textup{ Otherwise.}
    \end{cases}
\end{align}
\end{lemma}

% \begin{proof}
% If $\bar q \divides n$, then $\langle r,\frac{n}{\bar q}\rangle$ is an integer, so the equality follows. 
% So assume $\bar q \not \divides n$, and notice that for every $s \in \Z[i]$ we have 
% \begin{align*}
%     e\left( \langle s,\frac{n}{\bar q}\rangle\right) \, \sum_{r\in {B_q}} e\left( \langle r,\frac{n}{\bar q}\rangle\right) 
%         & = \sum_{r\in {B_q}} e\left( \langle r,\frac{n}{\bar q}\rangle\right),
% \end{align*}
% which then concludes the proof.
% \end{proof}

Below, we divide by  $\bar d$, where $d$ is a divisor of $q$. 
\begin{corollary} \label{cor:Orthogonality}
 For $d \divides q $ we have
\begin{align}
\label{e:orthognality2}
    \sum_{r\in B_q} e\left( \langle r,\frac{n}{\bar d}\rangle\right) = \begin{cases}
    N(q) & \textup{ if } \bar d \divides n \\
    0 & \textup{ Otherwise.}
    \end{cases}
\end{align}
\end{corollary}
% \begin{proof}
% \begin{align}
% % \label{e:orthognality}
%     \sum_{r\in B_q} e\left( \langle r,\frac{n}{\bar d}\rangle\right) = 
%     \sum_{r\in B_q} e\left( \langle r,\frac{n\frac{\bar q}{\bar d}}{\bar q}\rangle\right) = N(q) \mathbf{1}_{\bar q \divides n\frac{\bar q}{\bar d}}  = N(q) \mathbf{1}_{\bar d \divides n}.
% \end{align}
% \end{proof} 

Another consequence of Lemma \ref{orthoidentity-lm} is a form of Parseval's Identity as stated below.

\begin{proposition}
For a function $f$ defined on $\Z[i]$ the Discrete Fourier Transform $\mathcal{F}_{q}$ defined in \eqref{e:FDq} satisfies
    \begin{align*}
        \sum_{n\in B_q} |f(n)|^2 = \frac{1}{N(q)} \sum_{x\in B_{\bar q}} |\mathcal{F}_{q} f (x)|^2.
    \end{align*}
\end{proposition}
Above, on the left we have $B _{q}$, and on the right $B_{\bar q} = \{\bar{n}:n\in B_q \} =\overline{B_{q}}$.  

% \begin{proof}
%     We have
%     \begin{align*}
%         \sum_{x\in B_{\bar q}} |\mathcal{F}_{q}  (x)|^2
%             & = \sum_{x\in B_{\bar q}} \sum_{k\in B_{q}} \sum_{n\in B_{q}} \overline{f(k) \, e(-\langle x , k/q\rangle) } f(n) \, e(-\langle x , n/q\rangle)
%             \\
%             & = \sum_{n\in B_{q}} f(n) \sum_{k\in B_{q}} \overline{f(k)} \sum_{x\in B_{\bar q}} e(\langle x , \frac{k-n}{q}\rangle )
%             \\
%             & = {N(q)} \sum_{n\in B_{q}} f(n) \sum_{k\in B_{q}} \overline{f(k)} \mathbf 1 _{q \divides k-n} 
%             \\
%             & = N(q) \sum_{n\in B_q} |f(n)|^2 
%     \end{align*}
% \end{proof}

The analog of Ramanujan's sum is
\begin{align}\label{def-generalizedareagausssum1} 
    \tau_{q}(x) 
        & := \sum_{n\in \mathbb A_q} e(\langle x, \tfrac{n}{q}\rangle).
\end{align}

\textbf{}\begin{lemma}\label{indicatoridentity-lm}
 Let $r\in \Z[i]$. 
\begin{align*}
        \mathbf{1}_{\gcd(\bar r, q)=1} =  \frac{1}{N(q)} \sum_{k\in B_q} e\left(\langle -r,\frac{k}{q}\rangle\right) \tau_{{\bar q}}( k).
\end{align*}
\end{lemma}

\begin{proof}
Note from the definition in \eqref{def-generalizedareagausssum1}, that 
\begin{equation}
    \tau_{{\bar q}}( k) 
    = \bigl\langle \mathbf{1}_{\gcd(\bar r, q)=1} ,
     e ( \langle \cdot , k/\bar q ) \bigr\rangle. 
\end{equation}
Then, the conclusion follows from general facts about orthogonal bases.

%This is an immediate result of \eqref{e:orthognality}. For any area $S \subseteq %\Z[i]$ we see that:
%\begin{align*}
%    \sum_{k\in B_q} e\left(\langle -r,\frac{k}{q}\rangle\right) \tau_{S,\bar{q}}( k) 
%        & = \sum_{k\in B_q} e\left(\langle -r,\frac{k}{q}\rangle\right) 
 %           \sum_{\substack{n\in S\\ (n,\bar q)=1}} e\left(\langle \frac{  n}{ \bar q} ,  k\rangle\right)
%        \\ 
%        & =\sum_{\substack{n\in S\\ (n,\bar q)=1}} \sum_{k\in B_q}  e\left(\langle \frac{{n}-r}{\bar{q}} ,k\rangle\right)
%        \\ 
%        & =  N(q) \sum_{\substack{n\in S\\ (n,\bar q)=1}} \mathbf{1}_{\bar q|n-r}.
%\end{align*}
%Therefore, by taking $S = B_{\bar q}$, there is a unique element $n \in B_{\bar %q} $ such that $n \equiv  r \mod \bar q $, so we have that 
%\begin{align*}
%    \sum_{\substack{n\in  B_{\bar q}\\ (n,\bar q)=1}} \mathbf{1}_{\bar q|n- r} 
%        =
%        \begin{cases}
%    1 & \textup{ if } (\bar r,  q) = 1 \\
%    0 & \textup{ Otherwise,}
 %   \end{cases}
%\end{align*}
%which gives the desired result.

\end{proof}

\begin{lemma}\label{lm:multiplicative}
For $x\in \Z[i]$ we have
\begin{align}
    \tau_{q}(x) &=
    \sum_{d \divides (q,\bar x)} \mu(q/d) N(d).
\end{align}
In particular, $\tau_{q}(q) = \phi(q) $ and $\tau_{q}(1) = \mu(q) $. In addition, if $(q, \bar x) = 1 $ then $\tau_{q}(x) = \mu(q) $.
\end{lemma}
\begin{proof}
Note that thanks to Corollary \ref{cor:Orthogonality} we have that if $d \divides q $ then
\begin{align*} 
    \sum_{\substack{k\in B_q\\ d|k}} e\left(\langle x,\frac{k}{q}\rangle\right) 
        & = \sum_{k\in B_q} e\left(\langle x,\frac{k}{q}\rangle\right) \mathbf{1}_{d|k}
        \\ 
        & = \sum_{k\in B_q} e\left(\langle x,\frac{k}{q}\rangle\right) 
        \frac{1}{N(q)}
        \sum_{r\in B_q} e\left(\langle -r,\frac{ k}{ d}\rangle\right) 
        \\ 
        & = \frac{1}{N(q)}
        \sum_{r\in B_q} \sum_{k\in B_q} 
        e\left(\langle \frac{x- r\frac{\bar{q}}{\bar d}}{\bar{q}},k \rangle\right) 
        \\ 
        & = \sum_{r\in B_q}  \mathbf{1}_{{\bar q}| x-  r\frac{{\bar q}}{{\bar d}}}
        \\ 
        & = N(\frac{q}{d}) \mathbf{1}_{\frac{ q}{d} \divides \bar x}.
\end{align*}
The last equality comes from counting the number of $r$'s in $B_q$ for which the indicator is non-zero. Let $\bar x = \frac{ q}{d} x' $. Then $\bar x - \bar r \frac{ q}{d} \equiv 0 \mod  q  \, \Leftrightarrow \, \frac{ q}{d} (x' - \bar r) \equiv 0 \mod \bar q$. This means that $ x' \equiv \bar r \mod  d $, which means that there exists a unique such $\bar r \in B_{ d} $ and $N(q/d)$ of them in $B_{ q}$.

Hence, and with the assistance of the Inclusion-Exclusion principle, we observe that
\begin{align}\label{e:multiplicative}
    \tau_{q}(x) &= \sum_{\substack{k\in B_q\\ (k,q)=1}}e(\langle x,\tfrac{k}{q}\rangle) 
    \\
    &= \sum_{d \divides q} \mu(d) \sum_{\substack{k\in B_q\\ d|k}} e(\langle x,\tfrac{k}{q}\rangle) 
    \\
    & = 
    \sum_{d \divides q} \mu(d) N(\tfrac{q}{d}) \mathbf{1}_{\frac{ q}{d} \divides \bar x}
    \\
    & =
    \sum_{d \divides q} \mu(q/d) N(d) \mathbf{1}_{d \divides \bar x }
    \\
    & =
    \sum_{d \divides (q, \bar x)} \mu(q/d) N(d).
\end{align}
\end{proof}

\begin{corollary}\label{crlr:multiplicative}
The function $ \tau_{q}(x)$ is multiplicative and we have that 
\begin{align}
    \tau_{q}(x) &= \mu\Big(\frac{q}{(q, \bar x)}\Big) \frac{\phi(q)}{\phi(\frac{q}{(q, \bar x)})}.
\end{align}
\end{corollary}
\begin{proof}
Let $d = (q,\bar x) $. A direct consequence of the Chinese Remainder Theorem gives us
\begin{align*}
    \tau_{q}(x) 
        & = \sum_{r \in \A _{ q} } e ( \langle r, \tfrac{x}{\bar q} \rangle ) 
        \\
        & = \sum_{m \in \A _{ q/d } } \sum_{k \in \A _{d} } e ( \langle k \frac{q}{d} + md, \frac{x}{\bar q} \rangle ) 
        \\
        & = \sum_{m \in \A _{ q/d } } e ( \langle   m, \tfrac{x}{\bar q / \bar d} \rangle)  \sum_{k \in \A _{d} } e ( \langle k  , \tfrac{x}{\bar d} \rangle ) 
        \\
        & = \tau_{q/d}(x) \, \phi (d) 
\end{align*}
and the result follows immediately from the previous lemma and the multiplicative properties of $\phi$.
\end{proof}

With these identities in mind, we prove Cohen's identity.

\begin{lemma} \label{xplusrestimate-lm}
For $x\in \Z[i]$, the following identity holds: 
\begin{align*}
        \sum_{r \in \A _{\bar q} } \tau_{q}(x+r)
            = \mu (q) \tau_{q}(x).
\end{align*}
\end{lemma}

\begin{proof}
We have
\begin{align*}
     \sum_{r \in \A _{\bar q} } \tau_{q}(x+r) 
     & =  \sum_{r\in B_{\bar q}} 1_{(  r,\bar q)=1} \tau_{q}(x+r)
     \\
      & =\sum_{r\in B_{\bar q}} \frac{1}{N(q) } 
        \sum_{k\in B_q} e\left( \langle   -r,\frac{k}{q}\rangle\right) \,  \tau_{\bar q}(k) \, \tau_{q}(x+r) 
     \\
     & = \frac{1}{N(q)}   \sum_{k\in B_q}\tau_{\bar q}(k) 
        \sum_{r\in B_{\bar q}}  e\left( \langle  -r,\frac{ k}{q}\rangle\right) \tau_{q}(x+r)
\end{align*}

Now we compute the inner sum.
\begin{align*}
    \sum_{r\in B_{\bar q}}  e\left( \langle  -r,\frac{ k}{q}\rangle\right) \tau_{q}(x+r)
    & = 
    \sum_{r\in B_{\bar q}}  e\left( \langle  -r,\frac{ k}{q}\rangle\right)  \sum_{n \in \A _{ q} } e\left(\langle \frac{  n}{ q} ,  x+r\rangle\right)
    \\ & = 
    \sum_{n \in \A _{q} }
    e\left(\langle \frac{  n}{ q} ,  x\rangle\right)
    \sum_{r\in B_{\bar q}}  e\left(\langle \frac{  n-k}{ q} ,  r\rangle\right)
    \\ & = 
    \sum_{n \in \A _{ q} }
    e\left(\langle \frac{  n}{ q} ,  x\rangle\right) N(q) \mathbf{1}_{ q \divides n-k}
    \\ & = 
    N(q) e\left(\langle \frac{  k}{ q} ,  x\rangle\right) \mathbf{1}_{(k,q) = 1}.
\end{align*}
In addition, note that we have if $(k,q)=1$ then $\tau_{\bar q}(k) = \tau_{\bar q}(1) $.
Indeed, this means that $\bar k$ is a unitary element modulo $\bar q$ so the set $\{ \bar k s \, : \, s \in B_{\bar q} \}$ is just a rearrangement of $B_{\bar q} $, and so
 \begin{align}\label{e:cyclicramanujan}
     \tau_{\bar q}(k)
        & = \sum_{s\in \A_{\bar q}} e\left(\langle k,\frac{s}{\bar q}\rangle\right) 
        \\
        & = \sum_{s\in \A_{\bar q}} e\left(\langle 1,\frac{\bar k s}{\bar q}\rangle\right)
        \\
        & =  \sum_{s\in \A_{\bar q}} e\left(\langle 1,\frac{ s}{\bar q}\rangle\right)
        \\ & = \tau_{\bar q}(1). 
 \end{align}
Therefore we have that
\begin{align*}
\sum_{r\in \A_{\bar q}} \tau_{q}(x+r) 
     & =
      \sum_{k\in B_q}\tau_{\bar q}(k) e\left(\langle \frac{  k}{ q} ,  x\rangle\right) \mathbf{1}_{(k,q) = 1}
      \\ & =
      \tau_{\bar q}(1)
      \sum_{k\in B_q}e\left(\langle \frac{  k}{ q} ,  x\rangle\right) \mathbf{1}_{(k,q) = 1}
      \\ & =
      \tau_{\bar q}(1) \tau_{ q}(x).
\end{align*}
Finally, $\tau _{\bar q}(1)= \mu (q)$.  
\end{proof}

\begin{corollary}\label{c:GaussianCohenCor}
For $x\in \Z[i]$ we have
\begin{align} \label{e:12}
   \bigg|  \sum_{r\in \A_{\bar q}} \tau_{q}(x+r) \bigg|
    & \ll N(\gcd(x,q)) \, N(q)^{\epsilon}.
\end{align}
\end{corollary}

\begin{proof}
By Lemma \ref{xplusrestimate-lm}, it suffices to find a bound for $\tau_{ q}(x)$.
The inequality follows immediately from Lemma \ref{lm:multiplicative} and Corollary \ref{crlr:multiplicative}
\begin{align*}
    |\tau_{q}(x)| \ll N\big(\gcd(\bar x,q)\big) {N(q)}^{\epsilon}.
\end{align*}
\end{proof}

Now we are ready to prove the Gaussian version of an inequality, originally due to Bourgain, involving the Ramanujan sums. 
It assures us that while the summands above can, in general, be as big as $q$, this happens infrequently as $x$ varies. 

\begin{proposition}\label{p:BourgainRamanujan}
For every $B, k>2$, integers $N > N _{k, B}$ and $Q \leq (\log N)^B$ we have
\begin{align} \label{e:BourgainRamanujan} 
   \frac{1}{N} \sum_{N(x)<N}    \bigg[  \sum_{q \colon q< Q}  \lvert \tau_{q}(x+r) \rvert \bigg]^k \ll Q^{k+\epsilon}.
 \end{align}
\end{proposition}

\begin{proof}
The proof of this proposition is inspired from a proof due to Bourgain \cite{MR1209299}.
In view of \eqref{e:12}, it suffices to show that 
\begin{equation}
\frac{1}{N} \sum_{N(x)<N} \biggl(\sum_{N(q)<Q} N(\gcd(x,q))\biggr)^k \ll Q^{k+\epsilon}.
\end{equation}
Expanding out the $k$th power, we have 
\begin{align*}
    \sum_{N(q_1),N(q_2),\cdots,N(q_k) < Q} \frac{1}{N}\sum_{N(x)<N} \prod_{i=1}^k N(\gcd(x,q_i)) 
\end{align*}
The function  $ x \to  \prod_{i=1}^k N(\gcd(x,q_i)) $ has period of $\mathfrak{L}:=N(\lcm(q_1,\cdots,q_k))$.  
We are free to restrict attention to the case in which $N$ is much larger than $\mathfrak{L}\leq Q^k$.  
Thus, the bound is reduced to establishing that 
\begin{align*}
    \sum_{N(q_1),N(q_2),\cdots,N(q_k) < Q} \frac{1}{\mathfrak{L}}\sum_{N(x)<\mathfrak{L}} \prod_{i=1}^k N(\gcd(x,q_i)) 
    \ll Q ^{k+ \epsilon }. 
\end{align*}
We will establish a bound for the inner sum, namely 
\begin{equation}
\sum_{N(x)<\mathfrak{L}} \prod_{i=1}^k N(\gcd(x,q_i)) \ll Q^k. 
\end{equation}
The sum above is multiplicative in the variables $q_1, \cdots, q_k$.   
So, specializing to the case of  $q_i=\rho^{r_i}$, we can estimate as follows. 
Assume that $r = \max r_{i} $. Then 
$\mathfrak{L} =  N(\rho^{r}).$
Let $\rho^a\| x$. There are 
$\mathfrak{L}  N(\rho^{-a}) = N(\rho^{r-a})$
such $x$ with $N(x)<\mathfrak{L}$. Also $N(\gcd(x,\rho^{r_i})) = N(\rho^{\min (r_i,a) })$. So
\begin{align} \label{e:MULTI}
    \sum_{N(x)< \mathfrak{L}} \prod_{i=1}^k N(\gcd(x,\rho^{r_i}))  &\ll N(\rho^{r-a})  \prod_{i=1}^{k} N(\rho^{\min (r_{i},a) }) 
\\ & \ll  \prod_{i=1}^{k}  N(\rho ^{r_i}). 
\end{align}
Above, we can assume that $r = \max r_i = r_1$, and then estimate  
\begin{equation}
N(\rho^{\min (r_{i},a) }) \leq 
\begin{cases}
N( \rho ^a)  & i=1
\\
N(\rho^{r_i})  & 2\leq i \leq k 
\end{cases}.
\end{equation}
Using \eqref{e:MULTI}, the multiplicative property implies that 
\begin{equation}
\sum_{N(x)<\mathfrak{L}} \prod_{i=1}^k N(\gcd(x,q_i)) \ll 
  \prod_{i=1}^{k}  N(q_i)  \leq Q^k. 
\end{equation}
Thus, to conclude the estimate, we note that 
\begin{align*}
    \sum_{N(q_1),N(q_2),\cdots,N(q_k) < Q} \frac{1}{\mathfrak{L}} \ll Q^{\epsilon}
\end{align*}
This completes the proof.
\end{proof}

\bigskip
\section{The Vinogradov Inequality}

We prove an analogue of the Vinogradov inequality for $\Lambda$  on $\Z[i]$. 
That is, we control the Fourier transform of averages of the von Mangoldt function over a sector \eqref{e:omega},  
at a point which is close to a rational with relatively large denominator.

\begin{theorem}\label{vingradovgaussian}
Let $\alpha\in \mathbb{C}$ and  $a,q\in \Z[i]$ with $0<N(a)<N(q) < \sqrt N $ $\gcd(a,q)=1$ and  $N(\alpha-\frac{a}{q})<\frac{1}{N(q)^2}$.  Fix $\omega \subset \mathbb T $.  For $N > N _{\omega }$, we have   
\begin{align*}
\sum_{\substack{ n\colon N(n)<N \\ \arg(n)\in\omega }} \Lambda(n)e(\langle n, \alpha\rangle) \ll  
\frac{N\log^2(N)}{N(q) ^{1/2}}  + N^{99/100}  .
\end{align*}
\end{theorem}

The remainder of the section is devoted to the proof. We collect some background material. 
The ring $\Z[i]$ is an Euclidean domain, which means that for each $a,b\in \Z[i]$, there exists a unique pair of $q,r$ such that
$$a=bq+r \textup{ where } {r\in B_b}$$
Note that if $N(r) =\frac{N(q)}{2}$, then \textbf{}{$r$ is on the edges of square $\langle b,ib\rangle$}, we pick the point with positive phase {on the line that contains origin}, i.e. $0\leq \arg(r)<\pi$.

Similar to the Farey dissection, we can find an approximation of $\xi$ by rational points in $\mathbb{Q}[i]$. In fact, we can say that for every $N(\xi)<1$, there exists $\frac{a}{q}\in \mathbb{Q}[i]$ such that 
$$N(\xi - \tfrac{a}{q}) < \frac{c(q)}{N(q)^2}$$
where $c(q) \approx 1$ is a constant that depends on $q$ but is uniformly bounded above and below. 
%The best uniform upper bound is $1/3$ \textcolor{red}{($1/3$ works, but not sure if it is the best)}. 
Thus, the assumption of the Theorem says that $\alpha $ is well approximated by $a/q$, and we apply the result above 
with $N(q)$ large.

Next, we recall the Prime Number Theorem for Gaussian integers.  
For any choices of $\omega \subset \mathbb T $,  integers $A$, and $N(q)<\log^{\frac{A}{ {2}}-1}N$, 
\begin{align}
\psi _{\Lambda}  (x,q,r,\omega ) 
    & := \sum_{\substack{n \equiv r \mod q \\ N(n) < x \\ \arg(n)\in\omega }} \Lambda (n) = \frac{\lvert \omega\rvert }{2\pi}  \frac{x}{\phi(q)} +O\left( \frac{x \, N(q)}{\log ^A x}  \right)
\label{landaupnt2}
\end{align}
the implied constants are absolute. 
The last two equations come from \cite{MR2061214}*{Theorem 5.36 \& following hint}, which provides us with a uniform property of the Prime Number Theorem.

The principal technical tool is this Lemma. 
  
\begin{lemma}\label{cong}
 Let $N$ be a large integer and  $m\in \Z[i]$ with  $N(m)>N^{1/100}$ or $m=0$. 
For large $T$, and  $N(\alpha-\tfrac{a}{q})\leq \frac{1}{N(q)^2}$ where $0<N(a)<N(q)$ and $(a,q)=1$ 
%\textcolor{red}{(or  $a\in \mathbb{A}_q$)}
\begin{align*}
B(T,N,\alpha,m)& :=\sum_{t\colon 0<N(t)<T}
\min\left(\frac{N}{N(t)},\left(\frac{N}{N(t)}\right)^{\frac{1}{4}}\frac{1}{N(\|(\bar{t}+m)\alpha\|)^{\frac{3}{4}}}\right) 
\\ & \ll 
\frac{N\log(T)}{N(q)}+N^{\frac{1}{4}} N (q)^{\frac 34}  \log N(q)+N(q)^{\frac{1}{4}}N^{\frac{1}{4}}T^{\frac{3}{4}}
+N^{99/100}.
\end{align*}
\end{lemma}

\begin{proof}
The proof proceeds by  case analysis based on the values of $h$ and $r$ 
introduced here. 
Since $\Z[i]$ is a Euclidean domain, we know that  $t+\bar{m}=\bar{h}\bar{q}+\bar{r}$, with $r\in B_q$. 
Let $\beta:=\alpha - \frac{a}{q}$. Then $N(\beta) <N(q)^{-2}$ and
$$\| (\bar{t} +m)\alpha\| =\bigl\lVert hq\beta+r\beta+\tfrac{ra}{q}\bigr\rVert. $$
So we can rewrite our sum in terms of $h$ and $r$ as follows. 
\begin{align*}
B(T,N,\alpha,m) = \sum_{N(h)<\frac{T}{N(q)}}\sum_{r\in B_q} \min\left(\frac{N}{N(t)},\left(\frac{N}{N(t)}\right)^{\frac{1}{4}}\frac{1}{N(\|(\bar{t}+m)\alpha\|)^{\frac{3}{4}}}\right),
\end{align*}
where on the right we understand that $t=t_{h,r}$ is given by 
\begin{equation}
    t+\bar{m}=\bar{h}\bar{q}+\bar{r}, 
\end{equation}

\smallskip 

\textbf{Case 1: $\bar{t}+ m=0$.} 
Hence $\bar{t}=-m$, where $m$ is fixed. 
There is only one term in the sum we are estimating.
Since $N(t)>0$, we see that $m\neq 0$, and so $N(t) > N ^{1/100}$. Then we use the trivial bound, which gives the contribution of at most $N^{99/100}$ and we see the inequality holds in the case that $\bar{t}+m=0$.

\smallskip 
\textbf{Case 2: $h=0$, and $0<N(r)<N(q)/10$.} 
By assumption $N(r\beta)\leq (2N(q)) ^{-1}$, 
so $4N(\| (\bar{t}+m)\alpha\|)\geq N(\|\frac{ra}{q}\|)-\frac{1}{2N(q)}$. 
Therefore
\begin{align}\label{e:interestcase1}
\sum_{0 < N(r) \leq N(q)/10 }
\left(\frac{N}{N(t)}\right)^{\frac{1}{4}}\frac{1}{N(\| (\bar{t}+m)\alpha\|)^{\frac{3}{4}}}
    &\ll \sum_{N(r)<N(q)}\frac{N^{\frac{1}{4}}}{N(r)^{\frac{1}{4}} \left( N(\|\frac{ra}{q}\|) - \frac{1}{4N(q)} \right)^{\frac{3}{4}}}
\end{align}
Denote $d_a(r) \equiv a r \pmod q$, so that 
  $N(\|\frac{ra}{q}\|) = \frac{ N(d_a(r))}{N(q)}$. 
  We  see that we can ignore $\frac{1}{4N(q)}$ in the denominator of the right hand side. So
\begin{align*}
\eqref{e:interestcase1}
    &\ll N^{\frac{1}{4}}\sum_{N(r)<N(q)}\frac{1}{N(r)^{\frac{1}{4}}N(\|\frac{ra}{q}\|)^{\frac{3}{4}}}
\end{align*}
Now,the map $r \to ra$ is a permutation on $B_q$, so that 
\begin{align*}
\eqref{e:interestcase1}
    &\ll N^{\frac{1}{4}}\sum_{N(r)<N(q)}\frac{1}{N(r)^{\frac{1}{4}}N(   \frac{d_a(r)}{q} )^{\frac{3}{4}}}
\\
& \ll N^{\frac{1}{4}} N (q)^{3/4} 
    \sum_{N(r)<N(q)}\frac{1}{N(r)^{\frac{1}{4}}N({d_a(r)\textbf{}})^{\frac{3}{4}}}
\\
& \ll N^{\frac{1}{4}} N (q)^{ 3/4}  
    \sum_{N(r)<N(q)}\frac{1}{N(r)}
\ll N^{\frac{1}{4}} N (q)^{\frac 34}  \log N(q).     
\end{align*}
Above, we have used the $\ell^{4}$---$\ell^{4/3}$ H\"older inequality, and 
\eqref{sumofsqestimate}.  That completes Case 2. 
\smallskip

\begin{figure}
    \centering

\begin{tikzpicture}[thick, scale=0.6]
    \draw[thick] (0,0) -- (15,0) -- (15,15) -- (0,15) -- cycle;
    
    \filldraw[black] (7,7) circle (1pt);
    \draw[blue,thick] (0,0) -- (7,7);
    \filldraw[black] (3,6.5) circle (1pt);
    \draw[blue,thick] (0,0) -- (3,6.5);
    \filldraw[black] (3.5,4) circle (1pt);
    \draw[blue,thick] (0,0) -- (3.5,4);
    \filldraw[black] (3.5,3) circle (1pt);
    \draw[blue,thick] (0,0) -- (3.5,3);
    \filldraw[black] (5.5,3) circle (1pt);
    \draw[blue,thick] (0,0) -- (5.5,3);
    \filldraw[black] (4.5,3) circle (1pt);
    \draw[blue,thick] (0,0) -- (4.5,3);
    
    \filldraw[black] (6.5,9) circle (1pt);
    \draw[blue,thick] (0,15) -- (6.5,9);
    \filldraw[black] (4.5,8) circle (1pt);
    \draw[blue,thick] (0,15) -- (4.5,8);
    \filldraw[black] (3.5,11.5) circle (1pt);
    \draw[blue,thick] (0,15) -- (3.5,11.5);
    \filldraw[black] (2,9) circle (1pt);
    \draw[blue,thick] (0,15) -- (2,9);
    \filldraw[black] (6,12) circle (1pt);
    \draw[blue,thick] (0,15) -- (6,12);
    
    \filldraw[black] (8,12) circle (1pt);
    \draw[blue,thick] (15,15) -- (8,12);
    \filldraw[black] (9,9) circle (1pt);
    \draw[blue,thick] (15,15) -- (9,9);
    \filldraw[black] (12,7) circle (1pt);
    \draw[blue,thick] (15,15) -- (12,7);
    \filldraw[black] (11,13) circle (1pt);
    \draw[blue,thick] (15,15) -- (11,13);
    
    \filldraw[black] (8.5,7) circle (1pt);
    \draw[blue,thick] (15,0) -- (8.5,7);
    \filldraw[black] (9,3) circle (1pt);
    \draw[blue,thick] (15,0) -- (9,3);
    \filldraw[black] (12,4) circle (1pt);
    \draw[blue,thick] (15,0) -- (12,4);
    \filldraw[black] (8,4) circle (1pt);
    \draw[blue,thick] (15,0) -- (8,4);
\end{tikzpicture}
\caption{The distances  $\lVert \frac{ra}{q}+hq\beta+r\beta\rVert$ for a fixed $h$ are essentially uniformly distributed in case 3.}
\label{f:unidist}
\end{figure}
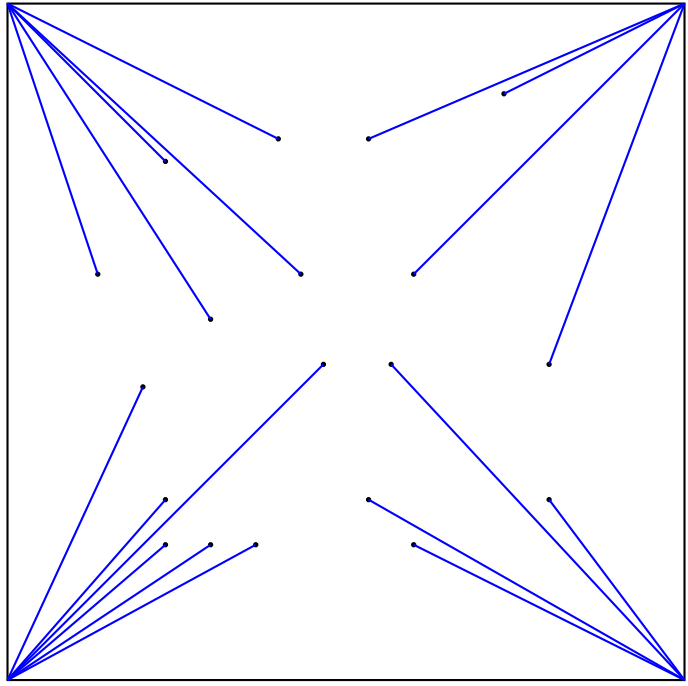

\textbf{Case 3.   $h\neq 0$, or  $h=0$ and $N(r)>\frac{N(q)}{10}$.}  
This is the principal case. 
Observe  that $N(t)\gg N(q)(N(h)+1)$. If $h \neq 0$ we can immediately see that 
$$ N(t) \gg N(h) \, N(q) - N(r) \geq N(h) \, N(q) - \frac{N(q)}{2} \gg N(q) \, (N(h) + 1), $$ 
or otherwise $N(t) = N(r) \geq {N(q)}/{10}$.  Combining the two inequalities we see that 
$$ N(t) \gg N(q)(N(h)+1).$$ 

The key claim comes in two parts: First, there are constants $0<C_1 < C_2$ and $C_3>0$ so that   for all $q$ and  $h$ as above, and  for all $r\in B_q$ there is a $d = d(r, h) \in B_q$ so that first, 
\begin{equation} \label{e:permutation} 
     C_1 N(\tfrac{d}{q}) \leq\bigl\lVert  \tfrac{ra}{q}+hq\beta+r \beta\bigr\rVert ^2 \leq C_2 N(\tfrac{d}{q}).
\end{equation}
and second, for all $d$, the  cardinality of $\{ r \colon d(r,h)=d\}$ is  at most $C_3$.   
(In this sense,   $\| \frac{ra}{q}+hq\beta+r \beta\|$ runs over the set $\|\frac{d}{q}\|$, as illustrated in  Figure \ref{f:unidist}.) 

\begin{proof} 
The middle term in \eqref{e:permutation} is at most one. 
We take $d = d(r,h) \in B_q$ to minimize the distance between $d/q$ and  the fractional part of $ \frac{ra}{q}+hq\beta+r \beta$.  Then, the upper and lower  bounds in 
\eqref{e:permutation} are immediate.

The remainder of the argument 
concerns the second claim above.  But that follows, since 
we always have for  $r_1\neq r_2 \in B_q$
\begin{align}
    N\left(\left\|  \frac{r_1a}{q}+hq\beta+r_1\beta \right\| - \left\|  \frac{r_2a}{q}+hq\beta+r_2\beta \right\| \right) & \gg  N\left(\frac{(r_1-r_2)a}{q}+(r_1-r_2)\beta  \right)
    \\
    & = 
    N \left(\frac{r_1 -r_2}{q} \alpha\right) \gg \frac{1}{N(q)} . 
\end{align}

\end{proof}

We can now turn to the sum $B(T,N, \alpha) $ in this case.  
Using the trivial bound for the cases that $r=0$ and picking the nontrivial bound for the other cases, we obtain
\begin{align*}
B(T,N,\alpha)
    & \ll \sum _{N(h) < \frac{T}{N(q)} } \sum _{ N(r) < N(q)/2 } \min \left( \frac{N}{N(q) \, (N(h) +1)} \right. ,
    \\
    & \hspace*{5cm} \left. \frac{N^{1/4}}{N(q)^{1/4} \, (N(h)+1) ^{1/4} \, N(\| hq\beta + \frac{ra}{q} + r\beta  \|)^{3/4} }\right) 
    \\
    & \ll \sum_{N(h)<\frac{T}{N(q)}} \Bigg( \Bigg.  \frac{N}{N(q)\left(N(h)+1\right)} 
    \\
    & \hspace*{3cm} +\sum_{0<N(r)<N(q)/2}\frac{N^{\frac{1}{4}}}{N(q)^{\frac{1}{4}}(1+N(h))^{\frac{1}{4}}N(\| hq\beta+\frac{ra}{q}+r\beta\|)^{\frac{3}{4}}}\Bigg.\Bigg)
    \\
    &\ll \sum_{N(h)<\frac{T}{N(q)}} \left(\frac{N}{N(q)\left(N(h)+1\right)}+\sum_{0<N(d)<N(q)/2}\frac{N^{\frac{1}{4}}}{N(q)^{\frac{1}{4}}(1+N(h))^{\frac{1}{4}}N( \frac{d}{q})^{\frac{3}{4}}}\right)
    \\
    &\ll \sum_{N(h)<\frac{T}{N(q)}} \left(\frac{N}{N(q)\left(N(h)+1\right)}+\frac{\,  N^{\frac{1}{4}}N(q)^{\frac{1}{2}}}{(1+N(h))^{\frac{1}{4}}}\sum_{k=1}^{N(q)}\frac{r_2(k)}{k^{\frac{3}{4}}}\right)
    \\ 
    &\ll \sum_{N(h)<\frac{T}{N(q)}} \left(\frac{N}{N(q)\left(N(h)+1\right)}+\frac{\,  N^{\frac{1}{4}}N(q)^{\frac{3}{4}}}{(1+N(h))^{\frac{1}{4}}}\right)
    \\
    &\ll  \sum_{0<\ell<\frac{T}{N(q)}} r_2(\ell) \left( \frac{N}{N(q)\ell} + \frac{ N(q)^{\frac{3}{4}} N^{\frac{1}{4}}}{\ell^{\frac{1}{4}}  } \right)
    \\
    & \ll \frac{N\log(T)}{N(q)}+N(q)^{  \frac{1}{4}}N^{\frac{1}{4}}T^{\frac{3}{4}}
\end{align*}
where we have used the estimates in \eqref{sumofsqestimate}, and our proof is complete.
\end{proof}

%%%%%%%%%%%%%%%%%%%%%%%%%%PROOF PROOF PROOF

Now are now ready to prove the main theorem of this section.

%%%%%%%%%%%%%%%%%%%%%%%%%%PROOF PROOF PROOF
\begin{proof}[Proof of Theorem \ref{vingradovgaussian}]
Vaughan's identity, well-known in the one dimensional case,  holds in two dimensions as well. That is, let $|f|\leq 1$ be an arithmetic function  and fix $UV<N$. 
Then 
\begin{align}\label{vaughan}
\sum_{N(n)<N}f(n)\Lambda(n)\ll U+\log(N)A_1(N,U,V)+N^{\frac{1}{2}}\log^3(N)A_2(N,U,V)
\end{align}
where $A_1$ and $A_2$  are given by 
\begin{align}  \label{A1}
A_1=&\sum_{N(t)\leq UV}\max_{1\leq w\leq N} \biggl\lvert \sum_{w\leq N(r)\leq \frac{N}{N(t)}}f(rt)\biggr\rvert\\
\label{A2} 
A_2=&\max_{U\leq M\leq N/V}\max_{V\leq N(j)\leq N/M}\left(\sum_{V<N(k)\leq N/M}\biggl\lvert
 \sum_{M<N(m)<\min(2M,\frac{N}{N(j)},\frac{N}{N(k)})}f(mj)\overline{f(mk)}\biggr\rvert\right)^{\frac{1}{2}}.
\end{align}

The term we need to estimate is \eqref{vaughan}, with $f(x)=e(\langle x,\alpha\rangle)$.
The two auxiliary integers are  $U=N^{\frac{1}{2}}$ and $V=N^{\frac{1}{4}}$.   
The first term requires the exponential sum estimate \eqref{e:expsum}, and 
 Lemma \ref{cong} in which we set $T=UV$. It follows that 
\begin{align} \label{a1}
    A_1 
    & =\sum_{N(t)<UV}\max_{w<\frac{N}{N(t)}} \biggl\vert \sum_{\substack{w<N(r)<\frac{N}{N(t)} \\\arg(r)\in\omega }}e(\langle rt,\alpha\rangle)\biggr\vert
        \\
    & \ll\sum_{N(t)<UV} \min\biggl(\frac{N}{N(t)},\left(\frac{N}{N(t)}\right)^{\frac{1}{4}}\frac{1}{N(\|\bar{t}\alpha\|)^{\frac{3}{4}}}\biggr) 
     \\
    & \ll \frac{N\log(N)}{N(q)}  + N^{\frac{1}{4}} N(q)^{\frac{3}{4}} \log N(q) + N^{\frac{1}{4}}[UV]^{\frac{3}{4}}N(q)^{\frac{1}{4}}  +N^{99/100}
    \\
    & \ll \frac{N\log(N)}{N(q)}  \Bigl(1 + (N(q)/N)^{\frac 34}\log N(q) + \frac{N(q)^{\frac14}}{N^{\frac3{16}}} \Bigr) +N^{99/100}
    \\
    & \ll \frac{N\log^2(N)}{N(q)}  + N^{99/100}  
\end{align}
where in the last inequality we have used Lemma \ref{cong} with the hypothesis that  $m=0$.  
This bound meets the claimed bound in Theorem \ref{vingradovgaussian}.  

\smallskip 

The second term from Vaughan's identity \eqref{vaughan} is quadratic in nature. 
\begin{align*}
A_2 ^2 & = \max_{U<M<\frac{N}{V}}\max_{V\leq N(j)<N/M} \sum_{V<N(k)\leq N/M}\Biggl\vert
\sum_{\substack{M<N(m)<2M\\N(m)<N/N(k),N/N(j)  \\\arg(m)\in\omega}}e\left(mj-mk,\alpha\right)\Biggr\vert
\\
&\ll \max_{U<M<\frac{N}{V}}\max_{V\leq N(j)<N/M} \sum_{V<N(k)\leq N/M}\min\biggl(
M,\left(\frac{N}{N(k)}\right)^{\frac{1}{4}}N(\| (j-k)\alpha\|)^{-\frac{3}{4}}\biggr)
\end{align*}
where we have used Lemma \ref{expsum}. Now we use Lemma \ref{cong} for $m=j$. 
The conclusion of the Lemma applies since in the sum above, we have  
$
M \leq \frac{N}{N(k)} 
$. 
Therefore
\begin{align}\label{a2}
A_2&\ll \max_{U<M<\frac{N}{V}}\max_{V\leq N(j)<N/M} \left(M+\frac{N\log(N)}{N(q)}+\frac{N\, N(q)^{\frac{1}{4}}}{M^{3/4}}+\left(\frac{N}{M}\right)^{\frac{1}{4}}N(q) +N^{\frac{99}{100}}\right)^{\frac{1}{2}}\\
&\ll{\frac{N^{1/2}\log N}{N(q)^{1/2}}}+\sqrt{\frac{N}{V}} +  \frac{N^{\frac{1}{2}}N(q)^{\frac{1}{8}}}{ U^{3/8}} + \left(\frac{N}{U}\right)^{\frac{1}{8}}N(q)^{\frac{1}{2}} +N^{\frac{99}{200}} . 
\end{align}
The last inequality follow from our choice of $U=N^{\frac{1}{2}}$ and $V=N^{\frac{1}{4}}$, and completes the proof. 
\end{proof}
%%%%%%%%%%%%%%%%%%%%%%%%%%PROOF PROOF PROOF

%%%%%%%%%%%%%%%%%%%%%%%%%%SECTION SECTION SECTION
\bigskip 
\section{Approximating the Kernel}

\begin{figure}
    \centering
    \begin{tikzpicture}
       \draw[->, black] (-2,0) -- (5,0);
       \draw[->, black] (0,-0.5) -- (0,6);
       \draw[blue, very thick]  (0,0) -- (2,1) ;
       \draw[blue, very thick]  (0,0) -- (-1,2) ;
       \filldraw[blue] (0,0) circle (2pt);
       \draw[gray, very thick]  (2,1) -- (4,2) ;
       \draw[gray, very thick]  (2,1) -- (1,3) ;
       \filldraw[gray] (2,1) circle (2pt);
       \draw[dashed, gray, very thick]  (1,3) -- (3,4) -- (4,2) ;
       \draw[gray, very thick]  (-1,2) -- (-2,4) ;
       \draw[gray, very thick]  (-1,2) -- (1,3) ;
       \filldraw[gray] (-1,2) circle (2pt);
       \draw[dashed,gray, very thick]  (-2,4) -- (0,5) -- (1,3) ;
       \draw[black] (2,1) circle (2pt) node[anchor=west]{$q $};
       \draw[black] (-1,2) circle (2pt) node[anchor=east]{$iq $};
    \end{tikzpicture}
    \caption{\tiny \\ The decomposition of the complex field.}
    \label{img-decomposition}
\end{figure}

Define the approximating multiplier $L_N^{a,q}$ as follows:
\begin{align}  \label{e:LNaq}
    \widehat{L_N^{a,q}}(\xi) &:= \Phi(a,\bar q) \widehat{M_{N} ^{\omega }}(\xi - \frac{a}{q}). 
\end{align}
Above, we suppress the $\omega $ dependence in the already heavy notation, and we  use the notation 
\begin{align} 
\label{e:newgausssum}
    \Phi(a,q) &\coloneqq  \frac{\tau _{ q } (a)}{\phi(q)} ,
\end{align}
In addition, recall that $M_N = M_N ^{\omega }$ is an average over a sector defined by   a choice of interval $ \omega \subset \mathbb{T}$,
 see \eqref{e:MN}.
The weighted variant is 
\begin{equation}
 A_N ^{\omega } = A_N = 
 \frac{2\pi}  {\lvert\omega\rvert  N} \sum 
_{ \substack { N(n) < N \\\arg(n)\in\omega  }} \Lambda (n) \delta _n . 
 \end{equation}

\begin{lemma} \label{majorarcestimate}
For $ \omega \subset \mathbb T $,  $\alpha \in \mathbb{C} $ with  $N(\alpha) < 1 $, assume that there are $0 \leq N(a) < N(q) < Q $ such that $N ( \alpha - \frac{a}{q} ) < \frac{Q}{{N \cdot N(q) }} $ and  $a\in \mathbb{A}_q$). Then we have, for $A>1$,  and $N > N _{\omega ,A}$, 
\begin{equation}
\bigl\lvert \widehat{A_N ^{\omega }}(\alpha) - \widehat{ L_N^{a, q}}(\alpha ) \bigl\lvert \ll \frac{Q^4}{\log^A(N)}.
\end{equation}
\end{lemma}

The usual one dimensional approach to these estimates uses the  Prime Number Theorem, and Abel summation. 
Implementing that argument in the two dimensional case engages a number of complications.  
After all, the two dimensional Prime Number Theorem is adapted to annular sectors, whereas Abel summation is most 
powerful on rectangles in the plane.  We avoid these technicalities below. 
(Mitsui \cite{MR136590} sums over rectangular regions.)

\begin{proof}
The quantifications of the Prime Number Theorem are decisive. 
 Write 
\begin{align*}
 \widehat{A_N ^{\omega }}(\alpha) &= \frac{2\pi}  {\lvert\omega\rvert  N}  \sum_{\substack{N(n) < N \\ \arg(n)\in\omega } } \Lambda_{\Z[i]}(n) e(\left\langle n, \alpha \right\rangle)
 \\ & =
 \frac{2\pi}  {\lvert\omega\rvert  N} 
 \sum_{\substack{r \in B_{\bar q}  \\ (r,\bar{q})=1}} \; \sum_{\substack {n \equiv r \mod \bar{q} \\ N(n) < N  \\ \arg(n)\in\omega }}
  \Lambda(n) e(\left\langle n, \beta + a/q \right\rangle) 
 \\ & =
 \frac{2\pi}  {\lvert\omega\rvert  N} 
 \sum_{\substack{r \in B_{\bar q} \\ (r,\bar{q})=1}} e( \left\langle r,  a/q \right\rangle) 
 \sum_{\substack {n \equiv r \mod \bar{q} \\ N(n) < N \\ \arg(n)\in\omega }} \Lambda(n) e(\left\langle n, \beta  \right\rangle) 
 \\
 & = \frac 1 {\phi (q)}
 \sum_{\substack{r \in B_{\bar q} \\ (r,\bar{q})=1}} e( \left\langle r,  a/q \right\rangle)  
 B_N (r,\beta ), 
\end{align*}
where we define $B_N (r,\beta )$, and a closely related quantity by 
\begin{align}
B_N (r,\beta ) &\coloneqq 
\frac {2 \pi\phi (q) }{\lvert \omega  \rvert N} 
\sum_{\substack {n \equiv r \mod  \bar{q} \\ N(n) < N \\ \arg(n)\in\omega }} \Lambda(n) e(\left\langle n, \beta  \right\rangle) , 
\\
B'_{N} (r, \beta ) 
& \coloneqq 
\frac {2 \pi\phi (q) }{\lvert \omega  \rvert N} 
\sum_{\substack {n \equiv r \mod  \bar{q} \\ N(n) < N \\ \arg(n)\in\omega  }} e(\left\langle n, \beta  \right\rangle).
\end{align}

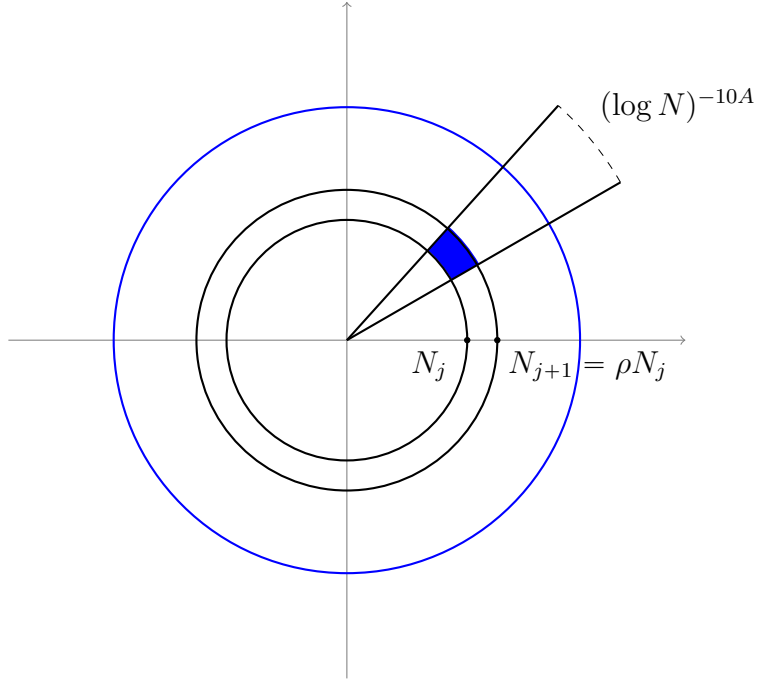
\begin{figure}
    \begin{tikzpicture}
        \draw[->, gray] (-4.5,0) -- (4.5,0);
        \draw[->, gray] (0,-4.5) -- (0,4.5);
        \draw[line width = 0.42cm, opacity=0.5, blue] (30:1.82cm) arc (30:48:1.82cm);
        \draw[dashed, line width = 0.3pt] (30:4.2cm) arc (30:48:4.2cm)
            node[anchor=west] {$ \quad {(\log N)^{-10A}} $};
        \draw[black,thick] (0,0) circle (1.6cm);
        \draw[black,thick] (0,0) circle (2cm);
        \draw[blue,thick] (0,0) circle (3.1cm);
        \draw[-, black,thick] (0,0) -- (30:4.2cm);
        \draw[-, black,thick] (0,0) -- (48:4.2cm);
        \filldraw[black] (1.6,0) circle (1pt) 
            node[anchor=north east] {$ N_j \; $};
        \filldraw[black] (2,0) circle (1pt) 
            node[anchor=north west] {$ N_{j+1} = \rho N_j $};
    \end{tikzpicture}
    \caption{\\\tiny figure description }
    \label{f:annular}
\end{figure}

Compare $B_{N,r}$ to $B'_{N,r}$, as follows.  Using the trivial estimate for 
$N(n) \leq  \sqrt N $
\begin{align}\label{e:differenceprimenonprime}
B _{N,r} (\beta ) -  B _{N}' (r,\beta ) 
\ll & 
N ^{-1/2} + 
\frac {2 \pi\phi (q) }{\lvert \omega  \rvert N} 
\sum _{\substack{ n \colon  \sqrt N < N(n) < N \\ n \equiv r \mod  \bar{q} \\ \arg(n)\in\omega }} 
\bigl( \Lambda(n) - \tfrac q {{\pi}\phi (q)}  \bigr) e(\langle n, \beta  \rangle). 
\end{align}
We continue with the last  sum above.  It is divided into annular rectangles, as follows.  
Let $\mathcal P$ be a partition of the arc $[0, \omega ] \subset \mathbb{T}$ into 
intervals of length approximately $(\log N)^{-10A}$.
Set $\rho = 1+ (\log N)^{-10A}$.   
For integers $j$ with 
\begin{equation}
N^{1/{2}} \leq N_j = \rho ^j {\sqrt{N}}< N , 
\end{equation}
and an interval $P\in \mathcal P$, set 
\begin{equation}
R (j, P) = \{ n \colon   N_j \leq N(n) < N _{j+1},\  \textup{arg}(n) \in P,\  n \equiv r \mod \bar q \}.  
\end{equation}
See Figure \ref{f:annular}. 
The set $R(j,P)$ is the symmetric difference of four sets to which the prime counting function estimate 
\eqref{landaupnt2} applies. 
From it, we see that 
\begin{align}\label{e:PNTHecke}
D(j,P)& =\sum _{n\in R(j,P) } 
\bigl( \Lambda(n) - \tfrac q {\phi (q)}  \bigr)  e(\langle n, \beta  \rangle)
\\
&\leq \sup_{n,m\in R(j,P)} \lvert 1 - e(\langle n -m, \beta  \rangle)  \rvert 
\sum _{n\in R(j,P) } \Lambda(n) + \tfrac q {\phi (q)}
\\
& \qquad 
+ \Bigl\lvert \sum _{n\in R(j,P) } \Lambda(n) - \tfrac q {\phi (q)} \Bigr\rvert 
\\
& \ll \Bigl[ \frac QN \cdot \frac N {\log^{10} N} \Bigr] ^{1/2} \lvert R(j,P) \rvert 
+ \frac{N_{j+1} - N_j } {(\log N)^{10A}}   
\\
&\ll  \sqrt Q  \frac{N_{j+1} - N_j } {(\log N)^{10A}}. 
\end{align}
The bound for the first term comes from the condition that $ N(\beta ) \leq \frac Q N$, and 
for the second from \eqref{landaupnt2}.

Control of the absolute value of the $D(j,P)$ is sufficient, since 
\begin{align}
   B _{N,r} (\beta ) -  B _{N,r}' (\beta ) 
& \ll 
N ^{-1/2} 
+ \frac{\phi(q)}N \sum_{ P \in \mathcal{P}} \sum_{ j \colon \rho^j \leq\sqrt N}  
\lvert D(j,P)\rvert  
\\
& \ll  N ^{-1/2} 
+ \frac{\phi(q)}N 
\sum_{ P \in \mathcal{P}} \sum_{ j \colon \rho^j \leq\sqrt N}  
 \frac{N_{j+1} - N_j } {(\log N)^{10A}}
 \\
 & \ll \frac{Q (\log N)^{A}} { (\log N)^{10A}} 
\end{align}
as there are only $\ll(\log N)^A$ choices of the interval $P$. 
We are free to choose $A$ as large as we want.  

This holds for all $r \in B _{\bar q}$, so that we have 
\begin{equation}
\widehat {A_N ^{\omega } } (\alpha ) - 
\frac 1 {\phi (q)}
 \sum_{\substack{r \in B_{\bar q} \\ (r,\bar{q})=1}} e( \left\langle r,  a/q \right\rangle)  
 B_N' (r,\beta )  
 \ll      \frac{Q} {(\log N)^{A}}. 
\end{equation}
Then, observe the elementary inequality that for $r, s \in B _{\bar q}$, 
\begin{align}
B_N' (r,\beta )  - B_N' (s,\beta ) & \ll \lvert r-s\rvert \cdot \lvert \beta\rvert  
\ll Q \Bigl[ \frac {Q}{  N} \Bigr] ^{1/2}   
\end{align}
which just depends upon the Lipschitz bound on exponentials, and the upper bound on $\beta $.

The conclusion of the argument is then clear. Up to an error term of magnitude $ Q ^{3/2}(\log N)^{A}$ we can write 
\begin{align}
\widehat {A_N ^{\omega } }(\alpha ) &=  
\frac 1 {\phi (q)}
 \sum_{\substack{r \in B_{\bar q} \\ (r,\bar{q})=1}} e( \left\langle r,  a/q \right\rangle)  
 B_N' (0,\beta )  
\\ &= \Phi (a, \bar q)B_N' (0,\beta )   
 \\
& =\Phi (a, \bar q)  \frac 1 {\lvert B _{\bar q} \rvert} \sum_{r\in B _{\bar q}} B_N' (r,\beta )    
\\
& = \Phi (a, \bar q) \widehat {M _N ^{\omega } } (\beta ). 
\end{align}
That is the conclusion of  Lemma \ref{majorarcestimate}. 
\end{proof}

Consider the following dyadic decomposition of rationals
\begin{equation}
\mathcal{R}_s = \left\lbrace  \frac{a}{q}\, : \, 2^s \leq N(q) < 2^{s+1}, \, a \in \A_q \right\rbrace.
\end{equation}
Let $\Delta $ be be a continuous  function on $\C$, 
a tensor product of piecewise linear functions, with 
\begin{equation}\label{eq-etadefinition}
\eta(\xi) = \begin{cases} 1 &\mbox{if }  \xi =(0,0)
\\
0 & \mbox{if } N(\xi) \geq\lVert \xi \rVert_\infty \geq 1 \end{cases}
\end{equation}
and let $\Delta _s(\xi) : = \eta(16^s \xi) $.   
Here, we remark that this definition is different from many related papers 
in the literature. With this defintion, the function $\check \eta $ is a tensor product F\'ejer kernels. In particular, they are 
non-negative averages.  Imposing this choice here will simplify considerations in the analysis of the Goldbach conjectures.

Recalling definitions of $L_N^{a, q}$ in \eqref{e:LNaq}, further define 
\begin{equation*}
\widehat{B ^{\omega }_N }(\xi):=  \sum_{s \geq  0}  \sum_{a/q \in \mathcal{R}_s }  \widehat{ L_N^{a, q}} \left( \xi \right) \Delta _s  \left( \xi - \frac{a}{q} \right).
\end{equation*}
We remind the reader that the $ \omega $ dependence is suppressed in the notation on the right.

%%%%%%%%%%%%%%%%%%%%%%%%%%THEOREM THEOREM THEOREM

\begin{theorem}\label{theorem:kernel_approximation}
Fix an integer $A >10$ and $\omega \subset \mathbb T $. 
 Then, there is an $N _{\omega }$ to that  for all $N > N _{\omega }$, 
\begin{equation}
\label{e:kernel_approx}
 \| \widehat{A ^{\omega }_N} -\widehat{B ^{\omega }_N } \|_\infty \ll (\log N)^{-A} . 
\end{equation}
The implied constant is independent of $\omega $. 

\end{theorem}
%%%%%%%%%%%%%%%%%%%%%%%%%%THEOREM THEOREM THEOREM

%%%%%%%%%%%%%%%%%%%%%%%%%%PROOF PROOF PROOF

\begin{proof}
A useful and familiar fact we will reference below is that for each $s$, the functions  below are disjointly supported. 
\begin{equation} \label{e:disjoint}
 \widehat{ L_N^{a, q}} \Bigl( \xi \Bigr) \Delta_s  \Bigl( \xi - \frac{a}{q} \Bigr),  
 \qquad a/q\in  \mathcal R_s
\end{equation}
Fix $\xi \in \mathbb T^2$.  
By the higher-dimensional 
Dirichlet's Theorem there are relatively prime numbers $a$ and $ q$ that satisfy $1 \leq N(a) \leq N(q) \leq N^{1/4} $ such that
\begin{equation*}
N \bigl( \xi - \tfrac{a}{q} \bigr) \leq \frac{1}{N(q) \cdot N^{1/4}}.
\end{equation*}
To prove the theorem we need to consider two cases based on the value of $N(q)$.  

\vspace{1em} 

\noindent

\textbf{Case 1: Suppose $1 \leq N(q) \leq (\log N)^{4A} $.} 
For $ \frac{a'}{q'}  \neq \frac aq $ and $N(q') \leq (\log N)^{2A}$, we have 
\begin{align*}
N \left( \xi - \frac{a'}{q'}   \right) 
    & \gg N \left( \frac{a'}{q'} - \frac{a}{q}  \right) - N \left( \xi -  \frac{a}{q}  \right) 
    \\
    & \geq \frac{1}{N(q')N(q)} - \frac{1}{N(q)} \frac{1}{N^{1/4}} \gg (\log N)^{-6A} . 
\end{align*}
Using this, and the decay estimate for $\widehat M_N$ in \eqref{e:expsum} to see that 
\begin{equation*}
\left|   \widehat{ L_N^{a', q'}} ( \xi  ) \Delta_s  \left( \xi - \frac{a'}{q'} \right) \right|  \ll \frac{1}{\sqrt{N(q')}} \left( N N \left( \xi - \frac{a'}{q'} \right) \right)^{-3/4} \ll N^{-1/2}.
\end{equation*}
Appeal to the disjointness property \eqref{e:disjoint}.  We have 
\begin{equation*}
\Biggl\lvert
  \sum_{s \colon 2^s \leq  (\log N)^{2A}}  \sum_{ \frac{a'}{q'} \in \mathcal{R}_s,  \frac{a'}{q'} \neq \frac{a}{q}}  \widehat{ L_N^{a', q'}} ( \xi ) \Delta_s  \left( \xi - \frac{a'}{q'} \right)
\Biggr\rvert 
   \ll N^{-1/2} \sum_{s \colon 2^s \leq   (\log N)^{2A}} 2^{-s} \ll N^{-1/2}. 
\end{equation*}
For $2^s > (\log N)^{2A}  $ we use the trivial bound  for $\widehat{M_N^\omega} $, the estimate for the 
Gauss sums in \eqref{e:newgausssum},  as well as the support property \eqref{e:disjoint}. This yields the estimate 
\begin{equation} \label{e:larges}
  \sum_{s \colon 2^s >  (\log N)^{2A} } 
    \sum_{ \substack{\frac{a'}{q'} \in \mathcal{R}_s \\ \frac{a'}{q'} \neq \frac{a}{q}}  }\widehat{ L_N^{a', q'}} 
    ( \xi) \Delta_s  \Bigl( \xi - \frac{a'}{q'} \Bigr) 
 \ll 
 \sum_{s \colon 2^s >  (\log N)^{2A} } 2^{- 3s/4} \ll (\log N)^{-A}.  
\end{equation}
Above, the sums exclude the case of $a'/q'= a/q$.  That is the central case, the  one 
Lemma \ref{majorarcestimate} was designed for.

We turn to the case of $a'/q'=a/q$ here. 
With appropriate choice of $Q$ and $A$ in that Lemma, we obtain from \eqref{e:kernel_approx}, 
\begin{align*}
\Bigl\lvert  \widehat{A_N ^{\omega }}(\xi) -  \widehat{ L_N^{a, q}} ( \xi ) \Delta _s  \Bigl( \xi - \frac{a}{q} \Bigr) \Bigr\rvert  
&\leq 
\Bigl\lvert  \widehat{A_N ^{\omega }}(\xi) -  \widehat{ L_N^{a, q}} ( \xi )  \Bigr\rvert  
+
\Bigl\lvert  \widehat{ L_N^{a, q}} ( \xi ) \Bigl(1-\Delta _s  \Bigl( \xi - \frac{a}{q} \Bigr)\Bigr) \Bigr\rvert  
\\
& \ll 
(\log N)^{-A} + \frac{N(q)^2}{ N(\xi -a/q)^{1/2}} \ll (\log N)^{-A}.
\end{align*}
This holds by choice of $a/q$. 
That completes this case.  

\vspace{1em}

\noindent
\textbf{Case 2: Suppose $ (\log N)^{4A} \leq N(q)   $}. 
Both terms are small.   By the Vinogradov inequality in Theorem \ref{vingradovgaussian} we have  
 \begin{align*}
    \lvert \widehat{A_N ^{\omega }}(\xi)  \lvert  \ll  (\log N)^{-A}. 
\end{align*}
It remains to show that $\widehat B_N(\xi ) $ is also small.  That function is a sum over integers $s \geq 0$.  
For $2^s > (\log N)^{2A}$, we only need to use the estimate \eqref{e:larges}.  
Thus, our focus turns to the case of $2^s \leq (\log N)^{2A}$.

For $2^s \leq (\log N)^{2A} $ we have for $ \frac{a'}{q'} \in \mathcal{R}_s $
\begin{align*}
    N \Bigl( \xi - \frac{a'}{q'}   \Bigr) 
        & \geq N \Bigl( \frac{a'}{q'} - \frac{a}{q}  \Bigr) - N \Bigl( \xi -  \frac{a}{q}  \Bigr) 
        \\
        & \geq \frac{1}{N(q')N(q)} - \frac{1}{N(q)} \frac{1}{N^{1/4}} 
        \\
        & \geq \frac{2^{-s-1}}{N(q)}\gg N^{-1/8}. 
\end{align*}
From the decay estimate in \eqref{e:expsum}, we have 
\begin{equation*}
\widehat{ L_N^{a', q'}} ( \xi) \Delta_s  \Bigl( \xi - \frac{a'}{q'} \Bigr) 
\ll \left( N N \left( \xi - \frac{a'}{q'} \right) \right)^{-3/4} \ll  N ^{-3/32}
\end{equation*}
Using the disjointness property \eqref{e:disjoint}, it then easy to see that 
\begin{equation}\label{case1}
\sum_{2^s \leq  \sqrt{Q}}  \sum_{ \frac{a'}{q'} \in \mathcal{R}_s,  \frac{a'}{q'} \neq \frac{a}{q}} 
 \widehat{ L_N^{a', q'}} ( \xi ) \Delta_s  \left( \xi - \frac{a'}{q'} \right) 
\ll (\log N)^{-A}. 
\end{equation}
That completes the second case, and hence the proof of our Theorem. 
\end{proof}

%%%%%%%%%%%%%%%%%%%%%%%%%%PROOF PROOF PROOF

\bigskip
\section{Estimates for the High and Low Parts}
Our High and Low decomposition of the multiplier incorporates a notion of smooth numbers.  
For integer $Q = 2^{q_0} \ll (\log N)^B$, we say that a Gaussian integer is $Q$-smooth if  $q$ is square free  and  the product of primes $\rho $ with $N(\rho )\leq Q$.  Here, $B$ will be a fixed integer.  

We write  $A_N ^{\omega } = \textup{Lo}_{Q,N} ^\omega + \textup{Hi} _{Q,N} ^{\omega }$, where 
\begin{align}  \label{e:Lo}
\widehat{\Lo}_{Q,N} ^{\omega }(\xi) &= \sum_{q \colon N(q) < Q} \sum_{a\in \mathbb{A}_q}
\Phi(a,\bar q) \widehat{M_N ^{\omega }} (\xi-\frac{a}{q}) \Delta_{q_0} (\xi-\frac{a}{q}).  
\end{align}
Here, we recall that $\omega\subset \mathbb{T}$ is an interval, and $N > N _ \omega$ 
is sufficiently large.   
The average $M ^\omega_N$ is defined in \eqref{e:MN}, the Gauss sum $\Phi(a,\bar q)$ in 
\eqref{e:newgausssum}.  (Note that the Gauss sum will be zero if $q$ contains a square.) 
This definition is inspired by Theorem \ref{theorem:kernel_approximation}.  But, the definition above incorporates 
not only smoothness, but the cutoff function  $\Delta_{q_0}$ is a function of $Q$. 
Both changes are useful in the next section. 
There are two key properties of these terms.  The first, is that the `High' part has small $\ell^2$ norm.  

\begin{lemma} \label{l:Hi}
For any $\epsilon >0$, $\omega \subset \mathbb T$, 
there in an $N_ \omega$ so that for all $N> N_\omega$, 
\begin{align} \label{e:Hi}
\| \Hi_{Q,N}^{\omega }\|_{\ell_2\rightarrow \ell_2} \ll Q^{-1+\varepsilon}.
\end{align}
\end{lemma}

\begin{proof}
The $\ell^2$ norm is estimated on the frequency side. 
By Theorem \ref{theorem:kernel_approximation},  the High term is a sum of three terms.  They are, suppressing the dependence on $\omega $, 
\begin{align}
\widehat{\Hi_{Q,N} ^1}(\xi) & = \sum_{Q<2^{s+1} }
\sum_{ \substack{2^s \leq N(q) < 2^{s-1} \\  N(q) \geq  Q }}
\sum_{ \frac{a}{q} \in \mathcal R_s} \Phi(a,\bar q)\widehat{M_N^{\omega }}(\xi-\frac{a}{q})\Delta_s(\xi-\frac{a}{q})
\\
\widehat{\Hi_{Q,N} ^2}(\xi) & = 
\sum_{s}
\sum_{ \substack{2^s \leq N(q) < 2^{s-1} \\ N(q) < Q }}
\sum_{ \frac{a}{q} \in \mathcal R_s} \Phi(a,\bar q) \widehat{M_N^{\omega }} (\xi-\frac{a}{q}) \{\Delta_{q_0} (\xi-\frac{a}{q}) - \Delta_s (\xi-\frac{a}{q})\}
\\ 
\widehat{\Hi_{Q,N} ^3}(\xi)   & \coloneqq  \widehat{A ^{\omega }_N} -\widehat{B ^{\omega }_N }.
\end{align}
We address them in reverse order. 
The last term is controlled by Theorem \ref{theorem:kernel_approximation}. It  clearly satisfies our claim \eqref{e:Hi}. 

In the $\widehat{\Hi_{Q,N} ^2}$, the key term is the last difference between $\Delta _{q_0}$ and $\Delta  _s$.  
In particular, if $\Delta_{q_0} (\xi) - \Delta_s (\xi) \neq 0$, 
we have $N(\xi)\gg 2 ^{-q_0} $. 
It follows that $\widehat{M_N^{\omega }} (\xi) $ is relatively small. 
This allows us to estimate 
\begin{align}
\lVert \widehat{\Hi_{Q,N} ^2}(\xi)  \rVert_\infty 
&\leq 
\sum_{s}
\sum_{ \substack{2^s \leq q < 2^{s-1} \\ N(q) < Q}}
\max_{ \frac{a}{q} \in \mathcal R_s} 
\bigl\lVert \Phi(a,\bar q) \widehat{M_N^{\omega }} (\xi-\frac{a}{q}) \{\Delta_{q_0} (\xi-\frac{a}{q}) - \Delta_s (\xi-\frac{a}{q})\} 
\bigr\rVert_\infty 
\\
&\ll 
\sum_{s}
2^{- 3s/4} \min\{1, (N2^{-q_0})^{-3/4} + N^{-1/2}  \} \ll Q^{-1}. 
\end{align}
Here, we have used the disjointness of support for the different functions, and the exponential sum estimate Lemma \ref{expsum}.

\begin{comment}
We then need to consider the sum when $\Delta _s (\xi -a/q)=0$ while $\Delta _{q_0} (\xi -a/q) \not= 0$.  Then, $ 2 ^{s+4} \geq Q ^{4}$. 
Assume in this paragraph that $ N(q) \leq N ^{1/2}$. 
That means in particular that $ N (\xi -a/q) > 2 ^{-s}  $.  Then, $\widehat{M_N^{\omega }} (\xi-\frac{a}{q}) \ll \textcolor{red}{N ^{-3/4}2^{3s/4}}$. 
And, for a fixed $\xi $ and $q$, we have 
\begin{equation}
\lvert  \{ a \in \mathbb{A}_q \mid \Delta _{q_0} (\xi -a/q) \not=0 \} \rvert \ll 
N(q)/Q . 
\end{equation}
Then, we have for fixed $\xi $, 
\begin{align}
\sum_{s \colon  Q^4 < 2 ^{s+4} < 2 ^{4}\sqrt N} 
\sum_{ \substack{2^s \leq q < 2^{s-1} \\ \textup{$q$ is $Q$-smooth}}}
&
\sum_{ \frac{a}{q} \in \mathcal R_s} 
\bigl\lvert \Phi(a,\bar q) \widehat{M_N^{\omega }} (\xi-\frac{a}{q}) \{\Delta_{q_0} (\xi-\frac{a}{q}) - \Delta_s (\xi-\frac{a}{q})\} 
\bigr\rvert  
\\
& \ll 
\sum_{s \colon  Q^4 < 2 ^{s+4} < 2 ^{4}\sqrt N}   N ^{-3/8} N(q) ^{-1+\epsilon } \frac{N(q)}Q  \ll N ^{-1/4}. 
\end{align}
\end{comment}

It remains to bound the  term $\widehat{\Hi_{Q,N} ^1}$. But the smallest denominator $q$ that we sum over satisfies at least $N(q)>Q$, 
so that a similar argument 
leads to  
\begin{align}
\lVert \widehat{\Hi_{Q,N} ^1} \rVert_\infty 
 &  \ll 
 \sum_{Q<2^s<N^{1/4}}
\sum_{ \substack{2^s \leq N(q) < 2^{s-1} \\ N(q) \geq Q}}
\max_{ \frac{a}{q} \in \mathcal R_s} \lvert \Phi(a,\bar q) \rvert
 \\
 & = \sum_{Q<2^s<N^{1/4}} 2^{-(1- \epsilon  )s} \ll Q^{-1+ \epsilon }. 
\end{align}
\end{proof}

We turn to the Low term.  It has an explicit form in spatial variables. %Recall that the inequality  \eqref{e:BourgainRamanujan} controls the arithmetic part of the expression below.  
%%%%%%%%%%%%%%%%%%%%%%%%% 
 \begin{lemma}\label{low}
For $x\in \Z[i]$ we have the equality below, in which recall that $2^{q_0} =  Q < (\log N)^B$. 
\begin{equation}  
\label{e:LoEquals}
    \textup{Lo}_{N,Q} ^\omega  (x) =  \left( M_N^{\omega } \ast \widecheck{\Delta_ {q_0}}\right)(x)
     \sum_{ q \colon N(q) < Q }\frac{\mu (q) \tau _q (x)}{ \phi (q)} 
     %\sum_{r \in \A_{\bar q} } \tau_{q}(x+r) .
\end{equation}
And, moreover, for all $\epsilon >0$, and non-negative $f$
\begin{align} \label{e:LOOO}
 \Lo_{Q,N} ^\omega f (x) \ll   Q ^{ \epsilon }  \bigl[ ( M_N^{\omega } \ast \widecheck{\Delta_ {q_0}}) \ast f  ^{1+ \epsilon } (x) \bigr] ^{1/(1+\epsilon )}. 
\end{align}
\end{lemma}
 
\begin{proof}
We have for fixed $q$, 
\begin{align} 
\sum_{a \in \A_q}\Phi(a,\bar q) \int_{D}\widehat{M_N^\omega}(\xi-\frac{a}{q})\Delta_{q_0}(\xi-\frac{a}{q})e(\langle x,\xi\rangle)d\xi
    & =\sum_{a \in \A_q}\Phi(a,\bar q) e(\langle x,\frac{a}{q}\rangle)\left( M_N \ast \widecheck{\Delta_ {q_0}}\right)(x)
    \\
    & =  \left( M_N^{\omega } \ast \widecheck{\Delta_ {q_0}}\right)(x) \frac{1}{ \phi (\bar q)} \sum_{a \in \A_q}\tau_{\bar q}(a) e(\langle x,\frac{a}{q}\rangle)
    \\ 
    & =  \left( M_N^{\omega } \ast \widecheck{\Delta_ {q_0}}\right)(x) \frac{1}{ \phi (q)} \sum_{r \in \A_{\bar q} } \tau_{q}(x+r) .
\end{align}
We then apply Lemma \ref{xplusrestimate-lm}, and sum over $Q$-smooth denominators $q$ to conclude the first claim \eqref{e:LoEquals}.

For the second claim,  we use \eqref{e:BourgainRamanujan} in the standard way.  Fix $ 2 ^{s-1} \leq Q$, and consider the operator $A$ with kernel 
\begin{equation}
A_s  (x) =   \left( M_N^{\omega } \ast \widecheck{\Delta_ {q_0}}\right)(x)  \sum _{q \colon 2 ^{s} \leq q < 2  ^{s}}  
\frac {\lvert \tau _q (x) \rvert}{\phi (q)} . 
\end{equation}
For an  integer $k > 2\epsilon ^{-1}$, and non-negative $f \in \ell^{1+ \epsilon }$,  we have 
\begin{align}
 A_s f (x) & = \sum_y  \left( M_N^{\omega } \ast \widecheck{\Delta_ {q_0}}\right)(y)  \sum _{q \colon 2 ^{s} \leq q < 2  ^{s}}   
 \frac {\lvert \tau _q (y) \rvert}{\phi (q)} f (x-y) 
\\
& \leq 
\Bigl[  
 \sum_y  \left( M_N^{\omega } \ast \widecheck{\Delta_ {q_0}}\right)(y)  f(x-y) ^{k/(k-1)} \Bigr] ^{(k-1)/k}
\\ 
\qquad & \times 
\Bigl[ \sum_y  \left( M_N^{\omega } \ast \widecheck{\Delta_ {q_0}}\right)(y)  
 \sum _{q \colon 2 ^{s} \leq q < 2  ^{s}}   
 \Bigl[ \frac {\lvert \tau _q (y) \rvert}{\phi (q)}\Bigr] ^{k}
  \Bigr] ^{1/k} 
\\ 
& \ll 2 ^{\epsilon s} 
\Bigl[  
 \sum_y  \left( M_N^{\omega } \ast \widecheck{\Delta_ {q_0}}\right)(y)  f(x-y) ^{k/(k-1)} \Bigr] ^{(k-1)/k}
. 
 \end{align} 
We sum this over $ 2 ^{s-1} < Q$ to complete the proof of \eqref{e:LOOO}.

\end{proof}

%%%%%%%%%%%%%%%%%%%%%%%%%%PROOF PROOF PROOF

\section{Improving Inequalities}

In this brief section, we establish the improving inequalities, namely Theorem \ref{t:fixedscale}.  And, list some additional results that we could establish.  
For the convenience of the reader, we restate the improving inequality here, in a slightly more convenient form for our subsequent discussion.  And, there is no loss of generality to reduce to the trivial case of a sector. 

\begin{theorem} \label{fixedscale} 
For all $N$, and $1< p \leq 2$, we have, for $\omega = \mathbb{T} $, and functions $f, g $ supported on $[0,\sqrt N]^2$, 
\begin{align*}
N^{-1} \langle A_N ^{\mathbb{T} }f,g\rangle \ll N ^{- 2/p}  \lVert f \rVert_p \lVert g \rVert_p. 
\end{align*}
\end{theorem}

\begin{proof}
As the angle $\omega = \mathbb{T} $, we suppress it in the notation. 
We must prove the inequalities  over the open range of $1 < p < 2$. So, it suffices to consider the case that 
$f = \mathbf{1}_F$ and $g=  \mathbf{1}_G$, for $F, G \subset \{ n \colon N(n) \leq N\}  $.  
Interpolation proves them as stated.

Dominating the von Mangoldt function by $\log N$, we always have 
\begin{align*}
\langle A_Nf,g\rangle \ll  N (\log N) 
\lvert F \rvert \cdot 
\lvert G \rvert.  
\end{align*}
We can then immediately deduce the inequality if 
\begin{equation}
   N^{-2} \lvert F \rvert  \cdot 
 \lvert G \rvert  \ll (\log N) ^{-2p'} 
\end{equation}

So, we assume that this inequality fails, which will allow us to use our High Low decomposition.  
Namely,  for $0< \epsilon  < 1/2$ sufficiently small, set 
\begin{equation}
Q ^{ \frac {2(1+ \epsilon)}{1- \epsilon } } \simeq \frac {N^2}  
{ \lvert F \rvert  \cdot 
 \lvert G \rvert } \ll (\log N) ^{2p'}.
\end{equation}
Write  $A_N = \textup{Hi}_{N,Q} + \textup{Lo} _{N,Q}$.  
Appealing to \eqref{e:Hi} for the High term, and \eqref{e:LOOO} for the Low term, we have 
\begin{align}
\langle\textup{Hi}_{N,Q} f, g   \rangle & \ll  Q ^{-1+ \epsilon } \lVert f \rVert_2 \lVert g \rVert_2 
\ll  Q ^{-1+ \epsilon } [ \lVert f \rVert_1 \lVert g \rVert_1] ^{1/2}, 
\\
\langle\textup{Lo}_{N,Q} f, g   \rangle & \ll  N Q ^{\epsilon } \lVert f \rVert_{1+ \epsilon } \lVert g \rVert_{1+\epsilon }.  
\end{align}
By choice of $Q$, the two upper bounds nearly agree and are at most 
\begin{equation}
\langle\textup{Lo}_{N,Q} f, g   \rangle 
\simeq N ^{-1 + 2 \epsilon  }\bigl[  \lvert F \rvert  \cdot 
 \lvert G \rvert] ^{1- 2 \epsilon }. 
\end{equation}
That is the desired inequality, for $p' = \frac {1+ 2  \epsilon } \epsilon $.  And so completes the proof. 

\end{proof}

The techniques developed to establish the improving inequality can be elaborated on to prove additional results. 
We briefly describe them here.  

\begin{enumerate}
    \item  An $\ell^p \to \ell^p$, for $1< p < \infty $ inequality for the maximal function 
    $ \sup_N \lvert A_N f  \rvert $. Compare to \cite{MR995574}.  

    \item  A $(p,p)$, $1< p <2$,  sparse bound for the maximal function. 
    Here we use the terminology of \cite{MR4072599}, for instance. The interest in the sparse bound is that it immediately implies a range of weighted inequalities.  

    \item One can establish pointwise convergence of ergodic averages. 
    Let $(T_1, T_2) $ be commuting invertible measure preserving transformations of a probability space $(X, \mu)$.  For all $1< p < \infty$, and $f\in L^p (X)$, the limit 
    \begin{equation}
        \lim_{N} \frac1 N \sum_{ N(n)<N} \Lambda (n) f(T^n x) 
    \end{equation}
    exists for a.e.$(x)$.  Here, $T^{a+ib}=T_1^aT_2^{b}$. 
    Compare to \cite{MR3646766}. 
\end{enumerate}
We have given references particular to the primes (in $\mathbb Z$). 
%The recent book by one of us \cite{MR4512201} is a comprehensive survey of discrete Harmonic Analysis. 

\section{Goldbach Conjecture} \label{s:Goldbach}

The purpose of this section is to prove analogues of the Goldbach Conjecture on the Gaussian setting. 
We recall some elementary facts about Gaussian primes.   We address a binary and ternary form of the Goldbach conjecture. 
The binary form is addressed  in density form. Namely, most even integers are the sum of two primes. On the other hand all sufficiently large odd integers are the sum of three primes. 
We further restrict the arguments of the integers to be in a fixed interval.

\subsection{The Binary Goldbach Conjecture} % (fold)
\label{sub:the_binary_goldbach_conjecture}

% subsection the_binary_goldbach_conjecture (end)
 The  Goldbach Conjecture states that every even Gaussian integer can be written as the sum of two primes. We prove a density version of this result. 
It  uses  the High/Low decomposition. 
Observe that 
\begin{equation}
A_N ^\omega \ast A_N ^\omega  (n)  = 
\frac {2 \pi} { \lvert \omega\rvert^2  N^2} 
\sum_{ \substack{N(m_1), N(m_2) <N \\ m_1 + m_2 =n \\ \arg(m_1), \arg(m_2) \in \omega}} 
\Lambda (m_1) \Lambda(m_2).   
\end{equation}
If the sum is non-zero,  $n$ can be represented as the sum of two numbers in the support of the von Mangoldt function  $\Lambda$ intersected with 
$S ^ \omega_N$, where we define  
\begin{equation}
S ^\omega_N = \{ n \colon N(n) <N,  \arg(n) \in \omega   \}. 
\end{equation}
The von Mangoldt function is supported on the Gaussian primes, and their powers. The powers have density less than $ \ll \sqrt N $.   
Thus, it suffices to establish that 
\begin{equation}
\lvert   \{ n \in S ^\omega_N \colon  \textup{$n$ even \& }  A_N  ^\omega\ast A_N^\omega  (n) = 0  \} \rvert \ll \frac N {(\log N)^B} . 
\end{equation}

Recall from \eqref{e:Lo}, that we can write 
\begin{equation} \label{e:GHiLo}
 A_N ^\omega =  {\textup{Hi} } +  {\textup{Lo} }. 
\end{equation}
This depends upon a choice of $Q = (\log N)^B$ for some sufficiently large power of $B$,
 and we suppress the dependence on $Q$, $N$ and $\omega$ in the notation. 
   Thus, we write 
\begin{equation} \label{e:A=LoHi} 
A_N \ast A_N  = \textup{Hi} \ast A_N 
+ \textup{Lo} \ast \textup{Hi} + \textup{Lo} \ast \textup{Lo}. 
\end{equation}
On the right, the last term that is crucial. 
We further write it as 
\begin{equation}
   \textup{Lo} \ast \textup{Lo} 
   = \textup{Main} + \textup{Error}.  
\end{equation}
Aside from the main term, everything is small.  The binary part of 
 Theorem \ref{t:Goldbach}easily follows from the Lemma below, 
in which we collect the required estimates.  

\begin{lemma} \label{l:Estimates} We have the estimates below, valid for all choices of $B>1$. 
Setting $\tilde N = \tilde N_B \coloneqq [N  (\log N) ^{(B-1)/2} ] ^{-1} $
\begin{align} \label{e:LoLo}
   \lvert\{ n \in S ^{\omega }_N  \setminus S ^{\omega }{3N/4}  \colon  n\ \textup{even}, \ \textup{Main} (n) < 100\tilde N  
   \} \rvert &\ll  N ^2 \tilde N  ,  
    \\  \label{e:Error}
    \lvert \{ 0< N(n) <N  \colon    \lvert \textup{Error} (n) \rvert >   \tilde N 
      \rvert 
    & \ll N ^2 \tilde N , 
    \\ \label{e:LoHi}
   \lVert \textup{Lo} \ast \textup{Hi} (n)  \rVert _{\ell^2} & \ll \tilde N  ^{1/2}, 
  \\ \label{e:AHi}
  \lVert \textup{Hi} \ast A_N (n)  \rVert _{\ell^2}& \ll  \tilde N  ^{1/2}.   
    \end{align}
\end{lemma}

\begin{proof}[Proof of Binary Goldbach Theorem.] 
It suffices to show that the set of even integers $x \in S ^{\omega }_N$ with $N(x)> 3N/4$ 
that are not the sum of two primes $p_1,p_2 \in S ^{\omega }_N$ has cardinality at most $\ll \frac N{(\log N)^{B/4}} $, 
for an integer $B>10$.   

The number of representations of $x$ as the sum of two primes is 
\begin{equation}
A_N \ast A_N(x) = \textup{Main} + \textup{Error} + 
\textup{Hi} \ast A_N 
+ \textup{Lo} \ast \textup{Hi} . 
\end{equation}
We apply the previous Lemma to these terms. Note that $ N ^2 \tilde N = N/(\log N)^{(B-1)/2}$.  
The principal term is $\textup{Main} $.  
By \eqref{e:LoLo} it satisfies the conclusion of the Theorem.  We need to see that the remaining terms are all small, off an exceptional set of size $ N ^2 \tilde N$.  For $\textup{Error}$, that is exactly the conclusion of 
\eqref{e:Error}.  For the next term, $\textup{Hi} \ast A_N $, observe that 
\begin{align}
\lvert \{n\in S _{N} ^{\omega } \colon  \textup{Hi} \ast A_N \gg \tilde N\} \rvert 
& \ll  
\tilde N ^{-2} \lVert \textup{Hi} \ast A_N  \rVert_2 ^2  
\\ 
& \ll \tilde N ^{-2} \cdot [ N (\log N) ^{4B}] ^{-1}  \ll N (\log N) ^{-B}.  
\end{align}
Here, we apply \eqref{e:AHi} with a larger value of $B$.   The last term is controlled in the same way by \eqref{e:LoHi}. 
\end{proof}

We focus on the first estimate above. The Main term is  
\begin{equation}
    \label{e:Main} 
    \textup{Main}(x) \coloneqq 
    \sum_{\textup{$q$ is $Q$-smooth}} \frac{1}{\phi(q)^2} \sum_{a \in \A_q} \tau_{q} (a)^2 e\bigl(\langle \tfrac{a}{q} ,x\rangle\bigr) 
      \int_{\mathbb{T}^2} \widehat{M_N^\omega}(\xi)^2 e\bigl(\langle \xi ,x\rangle\bigr) \; d\xi 
\end{equation}
Above, we say that $q$ is \emph{$Q$-smooth} if $q$ is square free and all prime factors $\rho $ of $q$ satisfy $N(q)< Q$.
The expression above can be calculated explicitly, 
using the Ramanujan like sum \eqref{def-generalizedareagausssum1}. 

\begin{lemma}\label{l:lowlowpart}
Recall that $2^{q_0} =  Q  < (\log N)^B$.  
For every $N(x) <N$ we have
\begin{align}  \label{e:lowlow}
     \textup{Main}(x)  =  M_N ^{\omega}  \ast  M_N ^{\omega}  (x)\sum_{\textup{$q$ is $Q$-smooth}} \frac{|\mu(q)|}{\phi(q)^2}  \tau_{q} (x), 
\end{align}
\end{lemma}
 
\begin{proof}
The term in \eqref{e:Main} is 
\begin{align*}
     \textup{Main} (x) &=    M_N(x) ^{\omega} \ast M_N ^{\omega}  (x)
    \sum_{\textup{$q$ is $Q$-smooth}} \frac{1}{\phi(q)^2} \sum_{a \in \A_q}  \tau_{q} (a)^2 e\bigl(\langle \tfrac{a}{q} ,x\rangle\bigr) 
\end{align*}
In the arithmetic term above, we fix $q$, expand the Ramanujan sums, and we use Lemma \ref{xplusrestimate-lm}.  This gives us 
\begin{align}
    \frac{1}{\phi(q)^2} \sum_{a \in \A_q} 
     \tau_{q} (a)^2 e\bigl(\langle \tfrac{a}{q} ,x\rangle\bigr) 
   &= 
    \frac{1}{\phi(q)^2} \sum_{a \in \A_q} 
   \Bigl[  \sum_{r\in \mathbb{A}_q} e\bigl(\langle \tfrac{a}{q} ,r\rangle\bigr)  \Bigr]^2 e\bigl(\langle \tfrac{a}{q} ,x\rangle\bigr)
     \\
    & = \frac{1}{\phi(q)^2} \sum_{r_1 \in \A_q} \sum_{r_2 \in \A_q} \tau_{q} (x+r_1+r_2) 
    \\
    &=\frac{1}{\phi(q)^2} \sum_{r_1 \in \A_q}  \tau_{q} (x+r_1)\tau_{\bar q} (1)
    \\  \label{e:tautau}
    & = \frac{1}{\phi(q)^2}  \tau_{q} (x)\tau_{\bar q} (1)^2 , 
\end{align}
where in the last line we have used Lemma \ref{xplusrestimate-lm} again. Finally $\tau_{\bar q} (1)^2 = \lvert \mu (q) \rvert $.  
\end{proof}

On the right in our equality for the Low-Low expression \eqref{e:lowlow}, the first 
convolution was analyzed in Lemma \ref{l:convolve}. We have 
\begin{align} 
\lvert  \{ x\in S ^{\omega }_N \colon N(x) \geq \tfrac 34 N ,\  & 
M_N ^{\omega }\ast  M_N ^{\omega}(n) \leq    [ N (\log N)^B] ^{-1}  
 \} \rvert 
 \\ 
 & \leq 
 \lvert  \{ x\in S ^{\omega }_N \colon  \textup{dist}(\textup{arg}(x), \mathbb{T} \setminus \omega ) 
 \leq C (\log N) ^{-B} \}\rvert 
 \\ 
 & \ll N (\log N) ^{-B}. 
\end{align}
We analyze the arithmetic part here. It does not require an exceptional set. 
And the Lemma below completes the proof of \eqref{e:LoLo}.  
Indeed, it is exactly this Lemma and its proof that motivates the use of the smooth numbers.

\begin{lemma}\label{l:HLConstantlemma}
We have for all even $x$, 
\begin{align} \label{e:HLConstantlemma}
    \sum_{\textup{$q$ is $Q$-smooth}} \frac{|\mu(q)|}{\phi(q)^2}  \tau_{q} (x) \gg  1
    \end{align}
\end{lemma}

\begin{proof}
Exploit the multiplicative structure of the sum.  One sees that it is a product over  primes. For any integer $x$, 
\begin{align}\label{e:finallyconstant}
    \sum_{\textup{$q$ is $Q$-smooth}} \frac{|\mu(q)|}{\phi(q)^2}  \tau_{q} (x) 
        & = \prod_{\rho \colon N(\rho ) < Q }  1 + \frac{\tau_{\rho } (x)}{\phi(\rho)^2} 
\end{align}
Above, the product is over primes $\rho $ with $N(\rho )< Q$, up to multiplication by units. 
 And recall that if $\rho\mid x$, 
we have $\tau_\rho(x)= \phi(\rho)$. 

It is important to single out the prime $1+i$. This is the unique prime with  $\phi (1+i)= 1$, and 
\begin{equation}
1 + \frac{\tau_{1+i} (x)}{\phi(1+i)^2} 
= \begin{cases}
2  & 1+i \mid x 
\\
0   & 1+i \nmid x 
\end{cases}
\end{equation}
Thus, if $1+i \nmid x $, the sum in \eqref{e:finallyconstant} is zero. 
But, evenness of $x$ is equivalent to $1+i \mid x$.  Thus, for even $x$, 
single out the case of $\rho=1+i$. 
\begin{align}
  \prod_{\rho \colon N(\rho ) \leq Q }  1 + \frac{\tau_{\rho } (x)}{\phi(\rho)^2}   
  &= 2\prod_{\substack{\rho|x,\ N(\rho ) \leq Q  \\ \rho\neq 1+i }} 
  \left( 1 + \frac{1}{\phi(\rho)}  \right)\left( 1 - \frac{1}{\phi(\rho)^2}  \right)^{-1}
   \times    \prod_{\substack{\rho,\ N(\rho ) \leq Q  \\ \rho\neq 1+i}} 1 - \frac{1}{\phi(\rho)^2} 
    \\
    & =
    2\prod_{\substack{\rho|x,\ N(\rho ) \leq Q  \\ \rho\neq 1+i}}  \frac{\phi(\rho)}{\phi(\rho)-1}  
   \times    \prod_{\substack{\rho,\ N(\rho ) \leq Q  \\ \rho\neq 1+i}}  1 - \frac{1}{\phi(\rho)^2}  
    \\
    & \coloneqq    h (x) \mathcal{G}. 
\end{align}
In the last line, we have written the product as a `local term' $h(x)$  and a `global term,' $\mathcal G$. 
If there is no $Q$-smooth prime that divides $x$, we understand that the local term is $2$. 
Thus, $h(x)$ is always at least 2 for even $x$.    The global term is finite: 
\begin{equation}
\mathcal G \geq \prod _{\rho }  1 - \frac{1}{\phi(\rho)^2}  ,
\end{equation}
and the infinite sum is convergent to a positive number, since 
\begin{align}
\sum_{\rho\ \textup{prime}} 
\frac{1}{\phi(\rho)^2}  & \ll \sum_{\rho\ \textup{prime}}  N(p)^{-2} \ll \sum_{k=1}^\infty k 2^{-k} < \infty.  
\end{align}
That completes our proof.  
\end{proof}

\begin{proof}[Proof of \eqref{e:Error}] 
We control the difference between the Low term and the Main term. This is the term $\textup{Error}$, and we only seek a distributional estimate on it. 
We begin by writing out this term explicitly. 
Recall that 
\begin{equation}
    \widehat{\Lo}(\xi) = \sum_{q \colon N(q)< Q} \sum_{a\in \mathbb{A}_q}
\Phi(a,\bar q) \widehat{M_N ^{\omega }} (\xi-\frac{a}{q}) \Delta_{q_0} (\xi-\frac{a}{q}).  
\end{equation}
By the disjointness of the supports of $\Delta _{q_0}( \cdot - \frac aq)$, for $N(q)< Q$, and the definition of the $\textup{Main}$ term in 
\eqref{e:Main}, we see that 
\begin{align}
\widehat{\Lo}(\xi) \cdot \widehat{\Lo}(\xi) 
& = \sum _{N(q)\leq Q} 
\sum _{a\in \mathbb{A} _q } 
\tau_{q} (a)^2 
       \Delta_ {q_0}(\xi - a/q)^2 \widehat{M_N^\omega}(\xi -a/q)^2 
\end{align}
We have done the work to invert this Fourier transform. In particular from \eqref{e:tautau}, we have 
\begin{align}
{\Lo}\ast {\Lo}(x) 
& =  M_N ^{\omega } \ast \check \Delta _{q_0} \ast  M_N ^{\omega } \ast \check \Delta _{q_0} (x)
 \sum _{q \colon N(q) <N } \frac{\lvert \mu (q) \rvert}{\phi(q)^2}  \tau_{q} (x)  . 
\end{align}

We can then explicitly write $\textup{Error}(x)$ as the sum of these three terms. 
\begin{align} \label{e:E1}
E_1 (x)  & \coloneqq \int (1 - \Delta _{q_0} (\xi ) ^2 ) \widehat{M_N ^{\omega }} (\xi) ^2 e ( \langle \xi ,x \rangle) \; d \xi
\sum _{q \colon N(q) < Q } \frac{\lvert \mu (q) \rvert}{\phi(q)^2}  \tau_{q} (x) , 
\\  \label{e:E2}
E_2 (x) & \coloneqq 
M_N ^{\omega } \ast M_N ^{\omega } (x) 
\sum _{q \colon Q \leq N(q) < N^{1/8} } \frac{\lvert \mu (q) \rvert}{\phi(q)^2}  \tau_{q} (x) , 
\\  \label{e:E3} 
E_3 (x) & \coloneqq 
M_N ^{\omega } \ast M_N ^{\omega } (x) 
\sum _{q \colon   N(q) \geq N^{1/8}  } \frac{\lvert \mu (q) \rvert}{\phi(q)^2}  \tau_{q} (x) . 
\end{align}
We address them in order.

\medskip 
For the control of $E_1$ defined in \eqref{e:E1}, 
an easily accessible $\ell^2$ estimate applies.
We recall that $Q \ll (\log N)^B$.  
By selection of $\Delta_{q_0}$ as a scaled version of a Fejer kernel, we have 
$1 - \Delta _{q_0} (\xi ) ^2 =O(Q^{-3})$  if $ N(\xi ) < Q^{-4}$.  So in this range we have
%\color{red}
\begin{align}
   \bigl\lVert 
\int_{N(\xi)<Q^{-4}} (1 - \Delta _{q_0} (\xi ) ^2 ) \widehat{M_N ^{\omega }} (\xi) ^2 e ( \langle \xi ,x \rangle) \; d \xi 
 \bigr\rVert _{\ell^2} \ll  Q^{-3}\int |\widehat{M_N^{\omega}}(\xi)|^2d\xi\ll N^{-1}Q^{-3}.
\end{align}
%\color{black}
But then it follows that 
\begin{equation}
\bigl\lVert 
\int (1 - \Delta _{q_0} (\xi ) ^2 ) \widehat{M_N ^{\omega }} (\xi) ^2 e ( \langle \xi ,x \rangle) \; d \xi 
 \bigr\rVert _{\ell^2} \ll  N ^{-1}Q^{-3}.
\end{equation}
On the other hand, it is easy to see that 
\begin{equation}
\sum _{q \colon N(q) < Q } \frac{\lvert \mu (q) \rvert}{\phi(q)^2}  \tau_{q} (x) \ll Q . 
\end{equation}
These two estimates prove the control required in \eqref{e:Error}. In particular we have 
\begin{equation}
   \lVert   E_1  \rVert _{\ell^2} \ll ( N^2 (\log N)^{B-1}) ^{-1/2}. 
\end{equation}
This estimate is stronger than than the required distributional estimate. %\textbf{THEN WHY DON'T WE JUST STATE IT?}

\smallskip 
We turn to the term $E_2$ defined in \eqref{e:E2}. 
For this and $E_3$, the leading 
term $M_N ^{\omega } \ast M_N ^{\omega } (x) \ll N ^{-1}$, so our focus is on the arithmetic terms.
The  definition $E_2$ requires that the denominators $q$ are at least $Q$, and less than $N^{1/8}$. 
The point is that the Ramanujan function $\tau _q$ is rarely more than $1$, as quantified by \eqref{e:BourgainRamanujan}.  
For integers $s$ with $Q/2 < 2^s \leq  N ^{1/8} $, we have 
\begin{equation}
\Bigl\lvert \Bigl\{  
N(x) < N \colon 
\sum _{ 2^s< N(q) \leq 2 ^{s+1}} \lvert \tau _q (x)\rvert 
> 2 ^{3s/2 } 
\Bigr\}\Bigr\rvert \ll N 2 ^{-3s}. 
\end{equation}
This follows from \eqref{e:BourgainRamanujan}, with $k=8$, and trivial bounds on the totient function. 
It is clear that this can be summed over these values of $s$ to complete the proof of \eqref{e:E2} in this case.  
Indeed, we have 
\begin{align}
\Bigl\lvert \Bigl\{  
N(x) < N \colon 
\sum _{ 2^s< N(q) \leq 2 ^{s+1}} \lvert \tau _q (x)\rvert 
> 2 ^{3s/2 }\Bigr\}\Bigr\rvert 
& \ll 
2 ^{-12s} \sum _{x \colon N(x)<N}
\Bigr[ 
\sum _{ 2^s< N(q) \leq 2 ^{s+1}} \lvert \tau _q (x)\rvert 
\Bigr] ^{8}
\\
& \ll 
2^{-3s} N. 
\end{align}
Using a trivial bound lower bound on the Totient function will complete this case.

\smallskip 
We turn to the term $E_3$ defined in \eqref{e:E3}. 
In this case, we require $N^{1/8} <N(q)$. 
And $N(q)$ can be as large and $e^{Q} = e^{c(\log N)^B}$.  
This  is too large to directly apply the previous argument. Instead, we will specialize the proof of 
\eqref{e:BourgainRamanujan} to this setting. 
For an integer $s$ with 
$N^{1/8} <2^{s+1}$, estimate 
\begin{align} \label{e:E31}
\sum _{\substack{q \colon 2^s\leq N(q) < 2 ^{s+1}
\\ q \textup{\ $Q$-smooth}}  } 
\frac{\lvert \mu (q) \rvert}{\phi(q)^2}  \tau_{q} (x)  & \ll 
s ^2  2 ^{-2s} 
\sum _{\substack{q \colon 2^s\leq N(q) < 2 ^{s+1}
\\ q \textup{\ $Q$-smooth}}  } 
N((x,q)) . 
 \end{align}
Above, we have used the familiar upper bound 
$\tau _q (x) \ll N((x,q)) $. 
Write $(x,q)=d$ and $q= q' d$.  
Continue 
\begin{align}
\eqref{e:E31} 
& \ll 
s ^2  2 ^{-2s} 
\sum_{d \textup{\ $Q$-smooth}} 
\mathbf{1}_{d\mid x } N(d)
\sum _{\substack{q' \colon 2^s\leq N(q')N(d) < 2 ^{s+1}
\\ q' \textup{\ $Q$-smooth}}  } 
\mathbf{1}
\\ 
& \ll s ^2 2 ^{-s}  
\sum_{d \textup{\ $Q$-smooth}}  \mathbf{1}_{d\mid x } 
 \ll  2 ^{-s/2}.
\end{align}
In the last line, we have $0< N(x) < N$, and $2^{s+1}> N^{1/8}$, 
so that we can use a favorable estimate on the divisor function.  This estimate is summable in $s$. It follows that 
for $0<N(x) < N$, that we have 
\begin{equation}
    E_3(x) \ll N ^{-17/16}. 
\end{equation}
This completes the analysis of the Error term.

\end{proof}

\begin{proof}[Proof of \eqref{e:LoHi}]  
We have from \eqref{e:Hi} and \eqref{e:LOOO},  
\begin{align}
\lVert \textup{Lo} \ast \textup{Hi}   \rVert _{\ell^2} ^2  
& = \int _{\mathbb{T}^2} \lvert  \widehat {\textup{Lo}}  \cdot \widehat{\textup{Hi}} \rvert ^2  \; d \xi 
\\
& \ll  Q ^{-1+ \epsilon }  \int _{\mathbb{T}^2} \lvert  \widehat {\textup{Lo}}  \rvert ^2  \; d \xi  
\\
& \ll  \frac {Q ^{-1+ 2\epsilon }} N . 
\end{align}
\end{proof}

The previous argument immediately implies the third and final estimate \eqref{e:AHi}. That completes the proof of  Lemma \ref{l:Estimates}, 
and hence the proof of our binary Goldbach Theorem.

\subsection{The Ternary Goldbach Conjecture} % (fold)
\label{sub:ternary_goldbach}

We turn to the ternary Goldbach conjecture, using the same notation.  We will show that for all  intervals $\omega \subset \mathbb T$, there is an $N_\omega $, 
so that for $N>N_\omega $, we have 
\begin{equation}  \label{e:TernaryGoldbach} 
 A_N  ^\omega\ast A_N^\omega \ast A_N^\omega  (n) \gg N^{-1}. 
\end{equation}
subject to the condition that $n \in T ^{\omega }_N$, where this condition means that 
\begin{equation} \label{e:Twn}
n \in T ^{\omega }_N \quad \textup{iff} \quad 
\begin{cases}
 n \in S ^{\omega }_N  \ \textup{is odd} 
 \\ 
 \textup{dist}(\textup{arg}(x), \mathbb{T} \setminus \omega ) > C (\log N)^{-B/3}  ,
 \\  
 N(n) > \tfrac 34 N.
\end{cases}
\end{equation}
 That is, every odd integer  in $S^\omega _N $ that is large enough and sufficiently far from the boundary,  has many representations as a sum of three primes, each of which is in the sector $\omega$. 
The middle condition involves an integer  $B>20$. 
It is motivated by Lemma \ref{l:convolve}. And, that Lemma indicates that the number of representations will depend upon the distance of $x$ to the boundary of the sector.

It remains to establish \eqref{e:TernaryGoldbach}.
We turn to the High/Low decomposition as in \eqref{e:GHiLo}, and write 
\begin{align}
A_N  ^\omega\ast A_N^\omega \ast A_N^\omega 
&= \textup{Hi} \ast A_N^\omega \ast A_N^\omega  + 
\textup{Lo} \ast A_N^\omega \ast A_N^\omega  
\\
&=\textup{Hi} \ast A_N^\omega \ast A_N^\omega  + 
\textup{Lo} \ast \textup{Hi} \ast A_N^\omega  + \textup{Lo} \ast \textup{Lo} \ast A_N^\omega   
\\  \label{e:Err3terms}
&=\textup{Hi} \ast A_N^\omega \ast A_N^\omega  + 
\textup{Lo} \ast \textup{Hi} \ast A_N^\omega  + \textup{Lo} \ast \textup{Lo} \ast \textup{Hi} 
+ \textup{Lo} \ast \textup{Lo} \ast \textup{Lo}
\\  
&\eqqcolon \textup{Err} + 
\textup{Lo} \ast \textup{Lo} \ast \textup{Lo}. 
\end{align}
As before, in \eqref{e:Main} we will write 
\begin{equation}\label{e:Main3}
\textup{Lo} \ast \textup{Lo} \ast \textup{Lo}
= \textup{Main}+\textup{Error}. 
\end{equation}
Thus, our focus is on the  Lemma below. It easily completes the proof. 

\begin{lemma}\label{l:3Gold} We have the estimates 
\begin{align}  \label{e:LoLoLo}
 [   N (\log N) ^{2B/3} ]  ^{-1} &\ll \min _{ \substack{ x\in T^\omega_N } } \textup{Main}(x), 
  \\ \label{e:Err}
   \lVert \textup{Error}  \rVert _{\ell^ \infty} +  
   \lVert \textup{Err}  \rVert _{\ell^ \infty} &\ll [ N (\log N)^{B-3}]^{-1}. 
\end{align}
\end{lemma}
Recalling the definition of $Q$-smooth from \eqref{e:Main}, 
we set 
\begin{equation} \label{e:ternary1}
\textup{Main}(x) \coloneqq 
\tilde M (x) \sum_{\textup{$q$ is $Q$ smooth}}\frac{\mu(q)}{\phi(q)^3}\tau_q(x), 
\end{equation}
where 
\begin{equation}
\tilde M(x) \coloneqq M_N^{\omega}\ast M_N^{\omega}\ast M_N^{\omega}\ast \widecheck{\Delta_{q_0}}\ast\widecheck{\Delta_{q_0}}\ast\widecheck{\Delta_{q_0}}. 
\end{equation}

\begin{proof}[Proof of \eqref{e:LoLoLo}] 
The details here are very close to those of the binary case.  
It is a consequence of Lemma \ref{l:convolve} that 
\begin{equation}
\inf _{x\in T^\omega _N} \tilde M(x) \gg 
[   N (\log N) ^{2B/3} ]  ^{-1}
\end{equation} 
 In the arithmetic part of \eqref{e:ternary1}, we have the M\"obius function $\mu (q)$, instead of $\lvert \mu (q) \rvert$ 
as in the binary case, and a third power of the totient function.  Using the multiplicative properties, we have 
\begin{align}
    \sum_{\textup{$q$ is $Q$ smooth}}\frac{\mu(q)}{\phi(q)^3}\tau_q(x)
    &=\prod_{N(p)<Q,p|x}\left(1-\frac{1}{\phi^2(p)}\right)\prod_{N(p)<Q,(p,x)=1}\left(1+\frac{1}{\phi^3(p)}\right)
    \\&\coloneqq h_3(x)\mathcal{G}_3
\end{align}
Parity is again crucial. We have 
\begin{equation}
1-\frac{\tau_{1+i}(x)}{\phi^3(1+i)} = 
\begin{cases}
0  & \textup{$x$ even}
\\
2  & \textup{$x$ odd}
\end{cases}
\end{equation}
That is, $h_3(x)$ is positive for odd $x$. 
Note that $\mathcal{G}_3=O(1)$, because
\begin{align}
    \mathcal{G}_3 = \prod_{N(p)<Q,(p,x)=1}\left(1+\frac{1}{\phi^3(p)}\right) < \sum_{n} \frac{|\mu(n)|}{N(n)^2} = O(1).
\end{align}
This completes the proof of \eqref{e:LoLoLo}.
\end{proof}

\begin{proof}[Proof of \eqref{e:Err}] 

There are two estimates to prove. 
The first is to bound the $ \ell^\infty$ norm of 
\begin{align}
\textup{Error} &\coloneqq 
\textup{Lo} \ast \textup{Lo} \ast \textup{Lo}
-\textup{Main} . 
\end{align}
Switch to Fourier variables. We have  
\begin{align}
\widehat{
\textup{Lo} }(\xi ) ^{3} 
= \sum _{q<Q} \sum _{a \in \mathbb{A} _q} 
\widehat{\tilde M}(\xi - a/q) \frac{\tau _q (a) ^{3}}{\phi(q)^3} . 
\end{align}
It follows that 
\begin{align}
\widehat{\textup{Error}} (\xi ) 
= \sum _{ \substack {q \textup{\ $Q$-smooth} \\ q\geq Q}}
 \sum _{a \in \mathbb{A} _q} 
\widehat{\tilde M}(\xi - a/q)\frac{\tau _q (a) ^{3}}{\phi(q)^3}. 
\end{align}
So, the $\ell^\infty$ norm of $\textup{Error}$ is at 
most the $L^1(\mathbb{T} ^2 )$ norm of the expression above. That is at most 
\begin{align}
\sum _{ \substack {q \textup{\ $Q$-smooth} \\ q\geq Q}}
\Bigl\lVert \sum _{a \in \mathbb{A} _q}
\widehat{\tilde M}(\xi - a/q)  &  \tau _q (a) ^{3} \Bigr\rVert _{L^1(\mathbb{T} ^2 )} 
\\
& \ll 
N ^{-1} 
\sum _{ \substack {q\geq Q}}
\phi (q) ^{-2+\epsilon }  \ll N ^{-1}  Q ^{-1+\epsilon } .
\end{align}
By our choice of $Q = (\log N)^B$, this estimate meets our requirements. 
\bigskip 

The second estimate concerns the term $\textup{Err}$. 
The term $\textup{Err} $ is a sum of three terms of the form $\phi _1 \ast \phi _2 \ast \phi _3$, 
where $\phi _j \in  \{  A_N, \textup{Hi}, \textup{Lo} \}$. And, at least one of the terms is a High term. See \eqref{e:Err3terms}.  We control each term.  To fix ideas, consider 
\begin{align}
    \lVert \Hi\ast A_N^{\omega}\ast A_N^{\omega}\rVert_{\infty} &\leq 
    \lVert \Hi\ast A_N^{\omega}\rVert_{2} \lVert  A_N^{\omega}\rVert_{2}
\\ 
&\leq \lVert \widehat{\Hi}\widehat{ A_N^{\omega}}\rVert_{2} \lVert  \widehat{A_N^{\omega}}\rVert_{2}
\\
    &\leq  \lVert \widehat{\Hi}\rVert_{\infty} \lVert  \widehat{A_N^{\omega}}\rVert_{2}^2
\end{align}
where the last inequality is trivial.
Now, $\lVert \widehat{\Hi}\rVert_{\infty} \ll Q^{-1+\epsilon}$, by \eqref{e:Hi}. Recall that $Q= (\log N)^B$, for an integer $B>10$. 
And, 
\begin{align}
\lVert \widehat{A_N^{\omega}}\rVert_{2}^2 
&= \lVert {A_N^{\omega}}\rVert_{2}^2 
\\
&\ll  N^{-2} \sum_{N(n)<N,\arg(n)\in \omega}\Lambda(n)^2 
\\ &\ll N^{-1}(\log N)^2 
\end{align}
We see that $ \lVert \Hi\ast A_N^{\omega}\ast A_N^{\omega}\rVert_{\infty}
\ll N^{-1} (\log N)^{-(1-\epsilon)B-2}$. That completes this case. 

The second term to control is 
\begin{align}
     \lVert \Hi\ast \Lo \ast \Lo \rVert_{\infty} 
     &\leq  \lVert \widehat{\Hi}\rVert_{\infty} \lVert  \Lo \rVert_{2}^2. 
\end{align}
To estimate this last term $ \lVert  \Lo \rVert_{2}^2$, use \eqref{e:LOOO}, which gives a better estimate than the first term. 
The third term to control is 
\begin{align}
     \lVert \Hi\ast A_N ^\omega  \ast \Lo \rVert_{\infty} 
     &\leq  \lVert \widehat{\Hi}\rVert_{\infty} 
     \lVert A_N^\omega \rVert_2
     \lVert  \Lo \rVert_{2}. 
\end{align}
But the right hand side is the geometric mean of the other two terms. That completes the proof. 

\end{proof}

\begin{comment}
\color{red}

\begin{proof}[Proof of Corollary \ref{c:bunyakowskicor}]

We prove the unconditional case. One can see that we can take $\omega$ to be small with respect to $N$ in the proof. For example proofs in this paper work for $\omega_N \simeq (\log N)^{-1/2}$. In other words, let $I=\frac{1}{8}+[-\omega_N,\omega_N]$, then we can prove that 
\begin{align}
    \left\{A_N^{I}\ast A_N(I)\right\} (x+a+i(x+b)) > \frac{x}{(\log x)^3}
\end{align}
 for almost every $x<\sqrt{N}$. It means tat there are at least $\omega_N x^2/(\log x)^6$ pairs of primes $p_1,p_2$ such that 
\begin{align}
    (x+a+i(x+b)) = p_1+p_2 
\end{align} 
and $\arg(p_1),\arg(p_2)\in I$. So we proved that there infinitely many Gaussian primes of the form 
$$x+c+i(x+d) \textup{ where } |c|,|d| < \frac{x}{(\log x)^{O(1)}}.$$
So we have infinitely many usual primes of the form 
\begin{align}
    (x+c)^(x+d)^2= \frac{1}{2}\left((2x+c+d)^2+(c-d)^2\right).
\end{align}
 One can take $m=c+d$ and $n=c-d$ to get the result.
\end{proof}

\color{black}
\end{comment}

% subsection ternary_goldbach (end)

\end{document}